\documentclass[11pt,twoside]{amsart}

\usepackage{latexsym}
\usepackage{amssymb}
\usepackage{vmargin}
\usepackage{amscd}
\usepackage{stmaryrd}
\usepackage{euscript}
\usepackage{mathrsfs}
\usepackage[all]{xy}
\usepackage{xr}
\DeclareMathAlphabet{\mathpzc}{OT1}{pzc}{m}{it}

\setmargins{32mm}{20mm}{14.6cm}{22cm}{1cm}{1cm}{1cm}{1cm}

\newcommand\da{\!\downarrow\!}

\newcommand\la{\leftarrow}

\newcommand\id{\mathrm{id}}

\newcommand\ten{\otimes}
\newcommand\hten{\hat{\otimes}}
\newcommand\vareps{\varepsilon}
\newcommand\eps{\epsilon}

\renewcommand\H{\mathrm{H}}
\newcommand\z{\mathrm{Z}}
\renewcommand\b{\mathrm{B}}

\newcommand\N{\mathbb{N}}
\newcommand\Z{\mathbb{Z}}

\newcommand\R{\mathbb{R}}
\newcommand\Cx{\mathbb{C}}

\newcommand\vv{\mathbb{V}}

\newcommand\bA{\mathbb{A}}

\newcommand\bE{\mathbb{E}}

\newcommand\bG{\mathbb{G}}

\newcommand\bP{\mathbb{P}}

\newcommand\bR{\mathbb{R}}

\newcommand\bT{\mathbb{T}}

\newcommand\C{\mathcal{C}}

\newcommand\cH{\mathcal{H}}

\newcommand\cL{\mathcal{L}}

\newcommand\cR{\mathcal{R}}

\newcommand\cT{\mathcal{T}}

\renewcommand\O{\mathscr{O}}

\newcommand\sA{\mathscr{A}}
\newcommand\sB{\mathscr{B}}

\newcommand\sF{\mathscr{F}}

\newcommand\sH{\mathscr{H}}

\newcommand\sO{\mathscr{O}}
\newcommand\sP{\mathscr{P}}
\newcommand\sQ{\mathscr{Q}}
\newcommand\sR{\mathscr{R}}

\newcommand\sT{\mathscr{T}}

\newcommand\sV{\mathscr{V}}

\newcommand\sZ{\mathscr{Z}}

\newcommand\fu{\mathfrak{u}}

\newcommand\Alg{\mathrm{Alg}}

\newcommand\Hom{\mathrm{Hom}}

\newcommand\End{\mathrm{End}}

\newcommand\Aut{\mathrm{Aut}}

\newcommand\Gal{\mathrm{Gal}}

\newcommand\im{\mathrm{Im\,}}

\newcommand\ch{\mathrm{ch}}

\newcommand\Ind{\mathrm{Ind}}

\newcommand\Set{\mathrm{Set}}

\newcommand\Sing{\mathrm{Sing}}
\newcommand\FD{\mathrm{FD}}

\newcommand\ad{\mathrm{ad}}

\newcommand\<{\langle}
\renewcommand\>{\rangle}
\newcommand\Lim{\varprojlim}
\newcommand\LLim{\varinjlim}

\newcommand\into{\hookrightarrow}

\newcommand\xra{\xrightarrow}

\newcommand\pr{\mathrm{pr}}

\newcommand\alg{\mathrm{alg}}

\newcommand\dmd{\diamond}
\newcommand\bt{\bullet}
\newcommand\by{\times}

\newcommand\Vect{\mathrm{Vect}}

\newcommand\Ban{\mathrm{Ban}}

\newcommand\Rep{\mathrm{Rep}}

\newcommand\SL{\mathrm{SL}}
\newcommand\SO{\mathrm{SO}}
\newcommand\GL{\mathrm{GL}}
\newcommand\SU{\mathrm{SU}}

\newcommand\Mat{\mathrm{Mat}}

\newcommand\an{\mathrm{an}}
\newcommand\hol{\mathrm{hol}}

\newcommand\PN{\mathrm{PN}}
\renewcommand\ss{\mathrm{ss}}
\newcommand\Irr{\mathrm{Irr}}
\newcommand\sgn{\mathrm{sgn}}
\newcommand\oHilb{\mathrm{Hilb}}

\newcommand\ev{\mathrm{ev}}

\newcommand\pro{\mathrm{pro}}

\newcommand\pd{\partial}
\newcommand\dc{d^{\mathrm{c}}}

\newcommand\tD{\tilde{D}}
\newcommand\tDc{\tilde{D^c}}

\newcommand\half{\frac{1}{2}}

\newcommand\st{\mathrm{st}}

\newcommand\ab{\mathrm{ab}}

\newcommand\gp{\mathrm{Gp}}

\renewcommand\alg{\mathrm{alg}}
\newcommand\red{\mathrm{red}}
\newcommand\Fr{\mathrm{Fr}}

\newcommand\row{\mathrm{row}}
\newcommand\Dol{\mathrm{Dol}}
\newcommand\dR{\mathrm{dR}}
\newcommand\DR{\mathrm{DR}}

\newcommand\co{\colon\thinspace}

\newcommand\oR{\mathbf{R}}

\newcommand\oT{\mathbf{T}}

\newcommand\uleft\underleftarrow
\newcommand\uline\underline
\newcommand\uright\underrightarrow

\newcommand\RFD{\mathrm{RFD}}

\newtheorem{theorem}{Theorem}[section]
\newtheorem{proposition}[theorem]{Proposition}
\newtheorem{corollary}[theorem]{Corollary}

\newtheorem{lemma}[theorem]{Lemma}
\newtheorem*{theorem*}{Theorem}
\newtheorem*{proposition*}{Proposition}
\newtheorem*{corollary*}{Corollary}
\newtheorem*{lemma*}{Lemma}
\newtheorem*{conjecture*}{Conjecture}

\theoremstyle{definition}
\newtheorem{definition}[theorem]{Definition}

\newtheorem*{definition*}{Definition}

\theoremstyle{remark}
\newtheorem{example}[theorem]{Example}

\newtheorem{remark}[theorem]{Remark}
\newtheorem{remarks}[theorem]{Remarks}

\newtheorem*{example*}{Example}
\newtheorem*{examples*}{Examples}
\newtheorem*{remark*}{Remark}
\newtheorem*{remarks*}{Remarks}
\newtheorem*{exercise*}{Exercise}
\newtheorem*{property*}{Property}
\newtheorem*{properties*}{Properties}

\begin{document}

\begin{abstract}
The pro-algebraic fundamental group can be understood as a completion with respect to finite-dimensional non-commutative algebras. We introduce finer invariants by looking at completions with respect to Banach and $C^*$-algebras, from which we can recover analytic and topological representation spaces, respectively. 
For a compact K\"ahler manifold, the $C^*$-completion also gives the natural setting for non-abelian Hodge theory; it has a  pure  Hodge structure, in the form of a pro-$C^*$-dynamical system. Its representations are pluriharmonic local systems in Hilbert spaces, and we study their cohomology, giving a principle of two types, and splittings of the Hodge and twistor structures.
\end{abstract}

\title{Analytic non-abelian Hodge theory}

\author{J.P.Pridham}
\thanks{This work was supported by  the Engineering and Physical Sciences Research Council [grant numbers  EP/F043570/1 and EP/I004130/1].}


\maketitle

\section*{Introduction}

In \cite{Sim1} and \cite{Sim2}, Simpson defined the coarse  Betti, de Rham and Dolbeault  moduli spaces of a smooth projective complex variety. These  are all algebraic spaces, with a  complex analytic isomorphism between the Betti and de Rham moduli spaces.
The non-abelian Hodge Theorem of \cite[Theorem 7.18]{Sim2} is a homeomorphism between the de Rham and Dolbeault moduli spaces, the key to which was the correspondence between semisimple local systems, pluriharmonic bundles and Higgs bundles.

The reductive pro-algebraic fundamental group $\pi_1(X,x)^{\red}$ introduced in \cite{Simpson} encapsulates, in a single object, all the information about the category of semisimple local systems. When $X$ is a compact K\"ahler manifold, the group scheme  $\pi_1(X,x)^{\red}$ also has a pure Hodge structure in the form of a discrete circle action, and a description in terms of Higgs bundles.

However, it has long been realised that  the reductive pro-algebraic fundamental group is slightly inadequate. From it we can recover the  points of the Betti moduli space, and from the full pro-algebraic fundamental group we can   even recover their infinitesimal neighbourhoods, but in general these groups convey no information about how the neighbourhoods glue together. A further source of dissatisfaction is the discontinuity of the circle action on $\pi_1(X,x)^{\red}$, since it is continuous on moduli spaces.  

The key idea behind this paper is that we can produce finer and more satisfactory invariants by looking at representations with analytic structure, an approach previously considered by Sullivan in \cite[\S 9]{Sullivan}. The group scheme $\pi_1(X,x)^{\red}$ can be recovered from representations in finite-dimensional matrix algebras, but the Riemann--Hilbert correspondence between Betti and de Rham moduli holds with coefficients in any Banach algebra. We accordingly construct Betti, de Rham and Dolbeault moduli functors on Banach algebras, and recover the analytic moduli spaces from these functors. The framed Betti and de Rham functors are represented by a Fr\'echet algebra which we regard as the analytic completion of $\R[\pi_1(X,x)]$.

To understand the topological structure underlying these analytic spaces, we then restrict to $C^*$-algebras rather than Banach algebras. There are notions of unitary and pluriharmonic representations with coefficients in any $C^*$-algebra, and the homeomorphism of moduli spaces above extends to an isomorphism between the  semisimple de Rham functor and the polystable Dolbeault functor on polynormal $C^*$-algebras, via isomorphisms with the pluriharmonic functor. The $C^*$-algebra of bounded operators gives us a notion of pluriharmonic local systems in Hilbert spaces, and there is a form of Hodge decomposition for these local systems. 

Lurking behind these comparisons is the twistor moduli functor on multiplicatively convex Fr\'echet $\O_{\bP^1}^{\hol}$-algebras. Its fibres at $\pm i \in \bP^1$ are the Dolbeault functor and its conjugate, while all other fibres are isomorphic to the de Rham functor. The Deligne--Hitchin twistor space can be recovered as an analytic space from the twistor moduli functor, and pluriharmonic torsors give a splitting of the  twistor moduli functor on $C^*$-algebras over $C(\bP^1(\Cx))$. 
Twistor cochains then admit a Hodge decomposition on pulling back along the Hopf fibration $SU_2 \to \bP^1(\Cx)$, 
and a continuous circle action serves to promote twistor structures to Hodge structures. 

The structure of the paper is as follows.
In \S \ref{prorepsn}, we cover some background material on pro-representability.  Proposition \ref{PNprop} then establishes a topological analogue of Tannaka duality for polynormal $C^*$-algebras and unitary representations, while Lemma \ref{PNlemma2} gives a similar result for non-unitary representations.  

In \S \ref{derhamsn}, we introduce the framed  Betti and de Rham functors $\oR^{B}_{X,x}$, $\oR^{\dR}_{X,x}$ on Banach algebras for any manifold $X$.
In Proposition \ref{cfBettiDR}, we establish an isomorphism $\oR^{B}_{X,x}(A)\cong\oR^{\dR}_{X,x}(A)$ for any  pro-Banach algebra $A$. We can even recover the analytic structure of moduli spaces of $G$-bundles from these symmetric monoidal functors (Remark \ref{recoveranalyticG}).
Proposition \ref{fgrep} then shows that $\oR^{B}_{X,x}$ is represented by a  Fr\'echet algebra completion $E^B_{X,x}$ of $\R[\pi_1(X,x)]$. This is a Fr\'echet bialgebra from which $\pi_1(X,x)$ can be recovered (Lemma \ref{recovergroup}).

In Definition \ref{pluritorsor}, we introduce a symmetric monoidal functor $ \oR^J_{X,x} $ on $C^*$-algebras, parametrising pluriharmonic bundles on a compact K\"ahler manifold $X$. This is representable by a pro-$C^*$-bialgebra $E^J_{X,x}$ (Proposition \ref{pluriprorep} and Lemma \ref{tensorstr}). There are also a symmetric monoidal functor
$\oR^{\Dol}_{X,x}$ on Banach algebras associated to Dolbeault moduli, which is seldom representable, and a harmonic functor extending the definition of $\oR^J_{X,x} $ to all Riemannian manifolds $X$, but with substantial loss of functoriality.

In \S \ref{cfansn} we establish relations between the various functors. 
We can recover the topology on moduli spaces of semisimple representations from  $E^J_{X,x}$ (Theorem \ref{Hsstopthm}). Proposition \ref{PNssprop} then gives a Tannakian description of  the polynormal completion of $E^J_{X,x}$, while Corollary \ref{PNrepss} gives a simple characterisation of continuous morphisms from $E^J_{X,x} $ to polynormal $C^*$-algebras. \S \ref{Dolsn} then gives similar results for $\oR^{\Dol}_{X,x}$.
 Lemma \ref{EJabgp} shows that  grouplike elements $G((E^J_{X,x})^{\ab})$  of the  abelianisation of $E^J_{X,x}$ are just $\H_1(X,\Z\oplus \R)$, with consequences for complex tori. 
 
There is a continuous circle action on $E^J_{X,x}$, so it is  a pro-$C^*$ dynamical system (Proposition \ref{circleX}). This allows us to regard $E^J_{X,x}$ as an analytic non-abelian Hodge structure of weight $0$ (Remark \ref{pureMHSrk}). In  Example  \ref{circleEJabgp}, we see that the circle action on  $G((E^J_{X,x})^{\ab})$ is just given by the Hodge structure on $\H^1(X,\R)$. Proposition \ref{VHSsemidirect} then characterises pure Hilbert variations of Hodge structure as representations of $E^J_{X,x}\rtimes S^1$.

\S \ref{cohosn} is concerned with Hilbert space representations of $E^J_{X,x}$, which  correspond to pluriharmonic local systems $\vv$ in Hilbert spaces. We can identify reduced cohomology $\bar{\H}^*(X,\vv)$ with the space of smooth $\vv$-valued harmonic forms, as well as establishing the principle of two types and a formality results (\S \ref{redcohosn}). There are analogous, but weaker, results for non-reduced cohomology (\S \ref{nonredcohosn}). 
The same is true of direct limits of Hilbert space representations, and Corollary \ref{univlocsys} shows that the universal such is the continuous dual $(E^J_{X,x})'$ (which can be regarded as the predual of the $W^*$-envelope of $E^J_{X,x}$).

In  \S \ref{twistorHodgecohosn}, these results are extended to show that the natural twistor structure on  $\bar{\H}^n(X,\vv)$ is pure of weight $n$ (Corollary \ref{harmcoho2}), with a weaker result for non-reduced cohomology (Proposition \ref{finecoho}). If $\bE$ is the local system associated to the $\pi_1(X,x)$-representation $ E^J_{X,x}$, then Proposition \ref{redenrich} shows that this twistor structure can be enhanced to a form of analytic Hodge filtration on  the de Rham algebra  $A^{\bt}(X, \bE')$. In \S \ref{SU2sn}, we re-interpret splittings of twistor structures and Archimedean monodromy in terms of the Hopf fibration.

Finally, in \S \ref{twistorfamilysn}, we introduce a whole twistor family of framed moduli functors on multiplicatively convex Fr\'echet algebras over $\bP^1(\Cx)$, from which we can recover both de Rham and Dolbeault functors. The coarse quotient of the framed twistor space is just the Deligne--Hitchin twistor space (Remark \ref{cfdh}). The twistor family carries a natural involution $\sigma$, and we show in Proposition \ref{sigmasections} that $\sigma$-equivariant sections of the framed twistor space are just framed pluriharmonic bundles. Theorem \ref{HTtopthm}, Proposition \ref{PNTprop} and Corollary \ref{PNrepT} then give analogues of Theorem \ref{Hsstopthm}, Proposition \ref{PNssprop} and Corollary \ref{PNrepss} in the twistor setting, describing twistors with $C^*$-algebra coefficients and the topology of the twistor space. 

I would like to thank Carlos Simpson for originally posing the problem of finding finer invariants, and for helpful discussions. I would also like to thank the anonymous referee for spotting erroneous statements about Fr\'echet algebras.

\subsection*{Notation}

We will use $k$ to denote either of the fields $\R, \Cx$.

\begin{definition}
 Given a $k$-Hilbert space $H$, write $L(H)$ for the space of $k$-linear bounded operators on $H$, with the norm topology.
\end{definition}

\begin{definition}
  Given topological spaces $X,Y$, we write $C(X,Y)$ for the set of continuous maps from $X$ to $Y$.
\end{definition}

\begin{definition}
 Given a group $G$ acting on a set $X$, write $[X/G]$ for the groupoid with objects $X$ and morphisms $X \by G$, where the source of $(x,g)$ is $x$ and the target is $xg$. Composition of morphisms is given by $(xg,h) \circ (x,g)= (x, gh)$.        
\end{definition}

\begin{definition}
 Given a group $G$ acting on sets $S,T$, write $S\by_GT$ for the quotient of $S\by T$ by the  $G$-action $g(s,t)= (gs,g^{-1}t)$.        
\end{definition}

\tableofcontents

\section{Pro-representability of functors on unital Banach algebras and $C^*$-algebras}\label{prorepsn}

\begin{definition}
 Given a functor $F\co \C \to \Set$, an object $A \in \C$ and an element $\xi \in F(A)$, we follow \cite[\S A.3]{descent} in saying that the pair $(A, \xi)$  is \emph{minimal} if for any pair $(A', \xi')$ and any strict monomorphism $f\co A' \to A$ with $F(f)(\xi')= \xi$, $f$ must be an isomorphism.

We say that a pair  $(A'', \xi'')$ is \emph{dominated} by a minimal pair if there exists a minimal pair $(A, \xi)$ and a 
morphism $g\co A \to A''$ with $F(g)(\xi)= \xi''$.
\end{definition}

\begin{definition}
As in  \cite[\S A.3]{descent}, we say that a functor $F\co \C \to \Set$ on a category $\C$ containing all finite limits is \emph{left exact} if it preserves all finite limits. This is equivalent to preserving finite products and equalisers, or to preserving fibre products and the final object.
\end{definition}

\begin{lemma}\label{dominatelemma}
Let $\C$ be a category containing finite limits, and  take a left exact functor $F\co \C \to \Set$.
Assume that for any cofiltered inverse system $\{A_i\}_i$ of strict subobjects of any object  $A\in \C$, the limit $\Lim_i A_i$ exists, and that the map
\[
 F(\Lim_i A_i) \to \Lim_i F(A_i)
\]
is an isomorphism. 

Then every pair $(A, \xi\in F(A))$ is dominated by a minimal pair.
\end{lemma}
\begin{proof}
Given the pair $(A, \xi)$, let $I$ be the full subcategory of the overcategory $\C \da A$ consisting of strict monomorphisms $B \to A$ for which $\xi$ lifts to $F(B)$. Note that this lift must be unique by the monomorphism property. 
If $f\circ g$ is a strict monomorphism, then so is $g$, which implies that all morphisms in $I$ must be strict monomorphisms in $\C$. Moreover, left-exactness of $F$ guarantees that $I$ is closed under the fibre product $\by_A$. the monomorphism properties imply that parallel arrows in $I$ are equal, so $I$ is a cofiltered category.

By hypothesis, the limit $L:= \Lim_{B \in I} B$ exists in $\C$. It is necessarily a strict subobject of $A$, since it is the limit of all parallel maps sourced at $A$ and  equalised by some $B \in I$. The unique lifts of $\xi$ to each $F(B)$ define an element of
\[
 \Lim_{B \in I} F(B),
\]
so by hypothesis we have a corresponding element $\eta \in F(L)$. Therefore $L$ is an object of $I$ and is in fact the initial object of $I$. The pair $(L, \eta)$ is therefore minimal, and dominates   $(A, \xi)$ as required.
\end{proof}

\begin{definition}
 Recall from \cite[\S A.2]{descent} that a pro-object $X \in \pro(\C)$ is said to be \emph{strict} if it is isomorphic to a pro-object of the form $\{X_i\}$, where each map $X_i \to X_j$ is an epimorphism.

A functor $F\co \C \to \Set$ is said to be \emph{strictly pro-representable} if there exists a strict pro-object $X$ with $F\cong \Hom_{\pro(\C)}(X,-)$.
\end{definition}

\begin{proposition}\label{prorepprop}
 Let $\C$ be a category containing finite limits and limits of cofiltered inverse systems of strict subobjects.
Then  a functor $F\co \C \to \Set$ is strictly pro-representable if and only if
\begin{enumerate}
 \item $F$ is left exact;
\item $F$ preserves limits of cofiltered inverse systems of strict subobjects.
\end{enumerate}
\end{proposition}
\begin{proof}
If $F$ satisfies the conditions, then it is left-exact, and by Lemma \ref{dominatelemma} every pair is dominated by a minimal pair. It therefore satisfies the conditions of   \cite[Proposition A.3.1]{descent}, so is strictly pro-representable.

Conversely, every pro-representable functor $F$ is left-exact, so we need only show that the second condition holds. Write $F= \LLim_{\alpha} \Hom(R_{\alpha},-)$ for a strict inverse system $\{R_{\alpha}\}_{\alpha}$, and take  a cofiltered inverse system $\{A_i\}_i$ of strict subobjects of some object  $A\in \C$. 

Given  an element $x \in  \Lim_i F(A_i)$ with image $x_i$ in $F(A_i)$, by definition there exist objects $R_{\alpha_i}$ and   maps   $y_i \co R_{\alpha_i} \to A_i$  lifting $x_i$. Now fix $i$; the liftings are compatible in the sense that for $j>i$ (increasing $\alpha_j$ if necessary), there is a commutative diagram
\[
\begin{CD}
 R_{\alpha_j} @>{y_j}>> A_j \\
@VVV @VVV\\
  R_{\alpha_i} @>{y_i}>> A_i.
\end{CD}
\]
Since $A_j \to A_i$ is a strict monomorphism and $R_{\alpha_j}\to R_{\alpha_i}$ an epimorphism, $y_i$ must lift to a map $y_{ij}\co R_{\alpha_i} \to A_j$. Since $A_j \to A_i$ is a monomorphism,  the lifting $y_{ij}$ is unique. Considering all $j>i$ together, this gives us a unique map $\tilde{x}\co R_{\alpha_i} \to \Lim_j A_j$, which gives rise to a unique pre-image $\tilde{x} \in F(\Lim_j A_j)$, as required.
\end{proof}

\subsection{Banach algebras}

\begin{definition}
Write $\Ban\Alg_k$ for the category of unital (not necessarily commutative) Banach algebras over $k$, with bounded  morphisms. 
\end{definition}

\begin{proposition}\label{repbanalg}
Take a functor $F\co \Ban\Alg_k\to \Set$ such that
\begin{enumerate}
 \item $F$ preserves all finite limits (equivalently: preserves fibre products and the final object), 
\item $F$ preserves monomorphisms (i.e. maps closed subalgebras to subsets), and 
\item for all inverse systems $S$ of closed subalgebras, the map
\[
 F(\bigcap_{s \in S} A_s) \to \bigcap_{s \in S} F(A_s)
\]
is an isomorphism.
\end{enumerate}
Then  $F$ is strictly pro-representable.
\end{proposition}
\begin{proof}
Given a Banach algebra $B$, every strict subobject of $B$ is a closed subalgebra. For any cofiltered inverse system $\{A_i\}_i$ of strict subobjects of any object  $B$, the limit $\Lim_i A_i$ exists in $\Ban\Alg_k$ and is given by $\bigcap_i A_i$. The result now follows from Proposition \ref{prorepprop}.
\end{proof}

\subsection{$C^*$-algebras}

\begin{definition}
Write $C^*\Alg_k$ for the category of unital (not necessarily commutative) $C^*$-algebras over $k$, with bounded involutive morphisms. 

Explicitly, a complex $C^*$-algebra is a complex Banach algebra equipped with an antilinear involution $*$ satisfying $(ab)^*=b^*a^*$ and $\|a^*a\|= \|a\|^2$ for all $a \in A$.

A real $C^*$-algebra is a real Banach algebra equipped with a linear involution $*$ satisfying the conditions above,  and having the additional property that $1+a^*a$ is invertible for all $a \in A$.
\end{definition}

A Banach $*$-algebra over $k$ is a $C^*$-algebra if and only if it is isometrically $*$-isomorphic to a self-adjoint norm-closed algebra of bounded operators on a Hilbert $k$-space; for $k=\R$, this is  Ingelstam's Theorem \cite[8.2 and 15.3]{goodearl}.

\begin{proposition}\label{repcstar}
Take a functor $F\co C^*\Alg_k\to \Set$ such that
\begin{enumerate}
 \item $F$ preserves all finite limits (equivalently: preserves fibre products and the final object), 
\item $F$ preserves monomorphisms (i.e. maps $C^*$-subalgebras to subsets), and 
\item for all inverse systems $S$ of nested $C^*$-subalgebras, the map
\[
 F(\bigcap_{s \in S} A_s) \to \bigcap_{s \in S} F(A_s)
\]
is an isomorphism.
\end{enumerate}
Then  $F$ is strictly pro-representable.
\end{proposition}
\begin{proof}
The proof of Proposition \ref{repbanalg} carries over.
\end{proof}

\begin{lemma}\label{consistentCstar}
 Every complex $C^*$-algebra becomes a real $C^*$-algebra by forgetting the multiplication by $\Cx$.       
\end{lemma}
\begin{proof}
We just need to show that  $1+a^*a$ is invertible for all $a \in A$. Now, $x \mapsto (1+|x|)^{-1}$ is a continuous function on $\R$ and $a^*a$ is positive self-adjoint, so the continuous functional calculus implies that $(1+a^*a)^{-1}   \in A$, as required. 
\end{proof}

\begin{lemma}\label{GalCstar}
The category $ C^*\Alg_{\R}$ is equivalent to the category of pairs $(A, \tau)$, for $A \in  C^*\Alg_{\Cx}$ and an involution $\tau\co A \to A$ 
  satisfying 
\begin{enumerate}        
\item $\tau(ab)= \tau(a)\tau(b)$, 
\item $\tau(a)^*= \tau(a^*)$, and 
 \item $\tau(\lambda)= \bar{\lambda}$ for $\lambda \in \Cx$.
\end{enumerate}
\end{lemma}
\begin{proof}
Given $B \in C^*\Alg_{\R}$, set $A:= B\ten_{\R}\Cx$; this is a complex $C^*$-algebra, with involution $(b\ten \lambda)^*= b^*\ten \bar{\lambda}$. The involution  $\tau$ is then given by complex conjugation, with $\tau( b\ten \lambda)= b\ten \bar{\lambda}$. 
For the quasi-inverse construction, we send a pair $(A, \tau)$ to the algebra $A^{\tau}$ of $\tau$-invariants. That this is a real $C^*$-algebra follows from Lemma \ref{consistentCstar}.

To see that these are quasi-inverse functors, first note that $(B\ten_{\R}\Cx)^{\tau}=B$. Next, observe that because $\tau$ is antilinear, we can write $A = A^{\tau} \oplus i A^{\tau} \cong A^{\tau}\ten_{\R}\Cx$ for all pairs $(A, \tau)$ as above.      
\end{proof}

\begin{definition}
For a complex (resp. real) $*$-algebra $A$, write $U(A)$  for the group of unitary (resp. orthogonal) elements 
\[
 \{a \in A\,:\, a^*a=aa^*=1\}.       
\]

Write $\fu(A)$  for the Lie algebra of anti-self-adjoint elements
\[
  \{a \in A\,:\, a^*+a=0\},      
\]
and write $S(A)$ for the self-adjoint elements of $A$, noting that $\fu(A)= iS(A)$ when $A$ is complex.
\end{definition}

%

\subsection{Representations and polynormal $C^*$-algebras}\label{PNsn}

\subsubsection{Representation spaces}

Fix a real unital $C^*$-algebra $A$.

\begin{definition}
 Write $\Rep_n^*(A)$  for the space  of  unital continuous $*$-homomorphisms $\rho \co A \to \Mat_n(\Cx)$  equipped with the topology of pointwise convergence.   Write $\Irr_n^*(A)\subset \Rep_n^*(A)$ for the subspace of irreducible representations. 
\end{definition}

\begin{definition}
Define $\FD\oHilb$ to be the groupoid of complex finite-dimensional Hilbert spaces and unitary isomorphisms. 
\end{definition}

\begin{definition}
 Write   $\FD^*\Rep(A)$ for the groupoid of pairs $(V, \rho)$ for $V\in \FD\oHilb$ and  $\rho \co A \to \End(V)$ a unital continuous $*$-homomorphism. Morphisms are given by unitary isomorphisms intertwining representations. The set of objects of $\FD^*\Rep(A)$ is given the topology of pointwise convergence.
\end{definition}

\begin{definition}\label{PNdef}
Define $A_{\PN}$ to be the ring of $\Gal(\Cx/\R)$-equivariant continuous additive endomorphisms of the fibre functor from $\FD^*\Rep(A)$ to the category of vector spaces. Explicitly, $a \in A_{\PN}$ associates to each pair $(V,\rho)$ an element $a(V, \rho) \in \End(V)$, subject to the conditions:
\begin{enumerate}
\item For any unitary isomorphism $u \co V \to W$, we have $a(W, u\rho u^{-1}) = u a(V,\rho) u^{-1}$. 
\item For any $(V_1,\rho_1), (V_2, \rho_2) \in \FD^*\Rep(A)$, we have $a(V_1\oplus V_2, \rho_1\oplus \rho_2) = a(V_1, \rho_1) \oplus a(V_2, \rho_2)$.        
\item The maps $a \co \Rep_n^*(A) \to \Mat_n(\Cx)$ are continuous and $\Gal(\Cx/\R)$-equivariant. 
\end{enumerate}
\end{definition}

\begin{lemma}\label{PNlemma}
The ring $A_{\PN}$ has the structure of a pro-$C^*$-algebra over $A$.  
\end{lemma}
\begin{proof}
We may describe $A_{\PN}$ as the categorical limit of a diagram of $*$-homomorphisms between the $C^*$-algebras $C(\Rep_n^*(A),\Mat_n(\R))$, thus making it into a pro-$C^*$-algebra. The $*$-homomorphism $A \to A_{\PN}$ is given by mapping $a$ to the transformation $a(V,\rho)= \rho(a)$.
\end{proof}

\subsubsection{Polynormal $C^*$-algebras}

\begin{definition}
Recall from \cite{pearcy} that an algebra $A$ is said to be $n$-\emph{normal} if for all $a_1, \ldots, a_{2n} \in A$, we have
\[
 \sum_{\sigma \in S_{2n}} \sgn(\sigma) a_{\sigma(1)} a_{\sigma(2)}\cdots  a_{\sigma(2n)}=0.      
\]
Call an algebra \emph{polynormal} if it is $n$-\emph{normal} for some $n$.
\end{definition}
By the Amitsur--Levitzki theorem, the  algebra  of $n \by n$-matrices over a commutative ring is $n$-normal. Also note  that an $n$-normal algebra is $k$-normal for all $k\ge n$.

Note that by restricting to $n$-dimensional representations, we get that for any real $C^*$-algebra $A$, the ring $A_{\PN}$ of Definition \ref{PNdef} is an inverse limit $A_{\PN} = \Lim_n A_{\PN,n}$ of $n$-normal $C^*$-algebras.

We now have a result combining aspects of Tannaka and Takesaki duality:

\begin{proposition}\label{PNprop}
If $A$ is a polynormal unital $C^*$-algebra, then the morphism $A \to A_{\PN}$ of Lemma \ref{PNlemma} is an isomorphism.
\end{proposition}
\begin{proof}
Since $A$ is $N$-normal for some integer $N$, \cite[\S 3]{pearcy} implies that $A$ is of type $I$, with all complex irreducible representations of $A$ having dimension at most $ N$. 

For a sufficiently large cardinal $\alpha$, \cite{bichteler} characterises the $W^*$-envelope $A''\ten \Cx$ of $A\ten \Cx$ as the ring defined analogously to $A_{\PN}$ by replacing ``continuous'' with ``bounded'' and ``finite-dimensional'' with ``of dimension at most $\alpha$''. Since all irreducible representations are at most $N$-dimensional,  the direct integral decomposition of $A$-representations gives us an injective map $A_{\PN} \to A''$, with boundedness following because any $a \in A_{\PN}$ is bounded on $\coprod_{k \le N}\Rep_k^*(A)$. 

Now, \cite{AkemannShultz} defines a ring $A_c \subset A''$ to consist of those $b$ for which the functions $b,b^*b,bb^*$ are weakly $*$-continuous on the space $P(A)$ of pure states of $A$. Since all irreducible representations arise as subrepresentations of $N$-dimensional representations, continuity on $ \Rep_N^*(A)$ suffices to give continuity on $P(A)$, so the inclusion $A_{\PN} \to A_c$ is an isomorphism.

By \cite{BunceDeddens}, the spectrum of $A$ is Hausdorff, and since $A$ is type $I$, \cite{AkemannShultz} then observes that $A$ is perfect, which means that the inclusion $A \to A_c$ is in fact an isomorphism. Thus the map $A \to A_{\PN}$ is an isomorphism.
\end{proof}

\subsection{The category of $C^*$-algebras with completely bounded morphisms}

\subsubsection{Basic properties}

\begin{lemma}\label{Cstarimage}
If $A$ is a $C^*$-algebra and $f\co A \to B$  a morphism of Banach algebras, then the image of $f$ has the natural structure of a $C^*$-algebra, with $f\co A \to \im(f)$ becoming a $C^*$-homomorphism.  
\end{lemma}
\begin{proof}
The kernel of $f$ is a closed two-sided ideal. Thus by \cite[Theorem 3]{segalIdeals},   
$A/\ker f$ is a $C^*$-algebra, as required.
\end{proof}

\begin{definition}\label{cbdef}
 Recall that a homomorphism $\pi\co A \to B$ of Banach algebras is said to be \emph{completely bounded} if 
\[
 \sup_{n \in \N} \|M_n(\pi)\|< \infty,
\]
where $M_n(\pi)\co M_n(A) \to M_n(B)$ is the morphism on $n \by n$ matrices given by $\pi$.

Given a pro-Banach algebra $A= \Lim_i A_i$ with completely bounded structure maps $A_j \to A_i$,  and a Banach algebra $B$, any morphism $\pi\co A \to B$ factors through some $A_i$, and we say that $\pi$ is \emph{completely bounded} if the map $A_j \to B$ is so for some  sufficiently large $j \ge i$. 

Write $\Hom(A,B)_{cb}$ for the set of completely bounded homomorphisms from $A$ to $B$.
\end{definition}

\begin{lemma}\label{KSPlemma}
 If $A$ is a $C^*$-algebra, then any      completely bounded homomorphism $f\co A \to L(H)$   is conjugate to a $*$-homomorphism of $C^*$-algebras. 
\end{lemma}
\begin{proof}
This is the main result of \cite{paulsen}.
\end{proof}

\begin{remark}\label{kadison}
Kadison's similarity problem asks whether all bounded (non-involutive) homomorphisms  between $C^*$-algebras are in fact completely bounded. The answer is affirmative in a wide range of cases, but the general problem remains open. 
Note that Gardner showed (\cite[Theorem A]{gardner}) that all Banach isomorphisms of $C^*$-algebras are conjugate to $C^*$-homomorphisms (and hence completely bounded). 

Also note that by \cite[Theorem 3]{segalIdeals}, every closed two-sided ideal of a $C^*$ algebra is a $*$-ideal; combined with Gardner's result, this implies that 
any bounded surjective map $A \to B$  between $C^*$-algebras must be conjugate to a $C^*$-homomorphism, hence completely bounded. The same is true of any bounded map between $C^*$-algebras whose image is a $C^*$-subalgebra. 
\end{remark}

\begin{definition}\label{cbalgdef}
Let $C^*B\Alg_k$ be the category of unital $C^*$-algebras over $k$, with completely bounded morphisms (which need not preserve $*$). 
 \end{definition}

\begin{lemma}\label{equivartpluri}
For  complex $C^*$-algebras $A,B$, giving a $U(B)$-equivariant function $f\co \Hom_{C^*\Alg_{\Cx}}(A,B) \to B$  (for the adjoint action on $B$) is equivalent to giving a $B^{\by}$-equivariant function $\tilde{f}\co \Hom_{C^*B\Alg_{\Cx}}(A,B) \to B$.
 \end{lemma}
\begin{proof}
 There is a canonical inclusion $\iota\co \Hom_{C^*\Alg_{\Cx}}(A,B)\to  \Hom_{C^*B\Alg_{\Cx}}(A,B) $, so given $\tilde{f}$, we just set $f$ to be $\tilde{f} \circ \iota$. 

The polar decomposition allows us to write $B^{\by}= B_{++}U(B)$, where  $B_{++}\subset S(B)$ is the subset  of strictly positive self-adjoint elements. Given $f$, there is thus an associated $B^{\by}$-equivariant function $\tilde{f}\co \Hom_{C^*\Alg_{\Cx}}(A,B)\by B_{++} \to B$ given by $\tilde{f}(p,g)= g^{-1}f(p)g$. By Lemma \ref{KSPlemma}, the map $ \Hom_{C^*\Alg_{\Cx}}(A,B)\by B_{++}\to  \Hom_{C^*B\Alg_{\Cx}}(A,B)$ is surjective, and we need to check that $\tilde{f}$ descends.

Now, if $\xi \in S(B)$ has the property that $\exp(\xi)$ fixes $p(A)$ under conjugation, then   $\exp(i\xi t)$ commutes with $f(p)$ for all $t$, so $i\xi$ must also. Thus $\xi$ and hence $\exp(\xi)$ commute with $f(p)$, so $\tilde{f}$ does indeed descend.
\end{proof}

\subsubsection{Representations}

We now fix a real unital $C^*$-algebra $A$.

\begin{definition}
Define $\FD\Vect$ to be the category of complex finite-dimensional vector spaces and linear maps. 
\end{definition}

\begin{definition}
 Write   $\FD\Rep(A)$ for the category of pairs $(V, \rho)$ for $V\in \FD\Vect$ and  $\rho \co A \to \End(V)$ a unital continuous morphism of Banach algebras. Morphisms are given by linear maps intertwining representations. The set of objects of $\FD\Rep(A)$ is given the topology of pointwise convergence.
\end{definition}

Note that the objects of $\FD\Rep(A)$ decompose into direct sums of irreducibles.

\begin{lemma}\label{PNlemma2}
The ring $A_{\PN}$ of Lemma \ref{PNlemma} is isomorphic to the ring $A_{\PN'}$ of $\Gal(\Cx/\R)$-equivariant continuous  endomorphisms of the fibre functor $\eta\co \FD\Rep(A) \to \FD\Vect$. Explicitly, $A_{\PN'}$ consists of elements $a$ such that
\begin{enumerate}
\item For any linear map  $f \co (V, \rho_1) \to (W, \rho_2)$, we have $a(W, \rho_2)f = fa(V,\rho_1)$. 
\item The maps $a \co \Rep_n^*(A) \to \Mat_n(\Cx)$ are continuous and $\Gal(\Cx/\R)$-equivariant. 
\end{enumerate}
\end{lemma}
\begin{proof}
First note that condition (1) applied to the projections $V_1\oplus V_2 \to V_i$ ensures that  for any $(V_1,\rho_1), (V_2, \rho_2) \in \FD^*\Rep(A)$, we have $a(V_1\oplus V_2, \rho_1\oplus \rho_2) = a(V_1, \rho_1) \oplus a(V_2, \rho_2)$.  

Restriction to $*$-representations then gives us a map $\psi\co A_{\PN'} \to A_{\PN}$. For a commutative $C^*$-algebra $C$, the $C^*$-algebra $\Mat_k(C)$ is of type $I$. This means that any  bounded map $A \to \Mat_k(C)$ is completely bounded, so taking $B= \Mat_k(C)$  in Lemma \ref{equivartpluri} for all $k$ ensures that $\psi$ is an isomorphism. 
\end{proof}


\begin{definition}\label{FDrepbase}
Given a commutative unital real $C^*$-algebra $A$ and a $*$-homomorphism $A \to B$ of real $C^*$-algebras,
 write   $\FD\Rep_{\hat{A}}(B)$ for the category of triples $(f,V, \rho)$ for $f \in \hat{A}$ (the spectrum of $A$), $V\in \FD\Vect$ and  $\rho \co B \to \End(V)$ a unital continuous morphism of Banach algebras for which $\rho(a)=f(a)\id$  for all $a \in A$.
Morphisms are given by linear maps intertwining representations. The set of objects of $\FD\Rep_{\hat{A}}(B)$ is given the topology of pointwise convergence.

The category $\FD\Rep_{\hat{A}}(B) $ has an additive structure over $\hat{A}$, given by $(f, V_1, \rho_1)\oplus (f,V_2, \rho_2)=(f,V_1\oplus V_2, \rho_1\oplus\rho_2)$.
\end{definition}

\begin{lemma}\label{PNlemma2base}
Given a commutative unital $C^*$-algebra $A$ and a $*$-homomorphism $A \to B$ of real  $C^*$-algebras,
the ring $B_{\PN}$ of Lemma \ref{PNlemma} is isomorphic to the ring  of $\Gal(\Cx/\R)$-equivariant continuous  endomorphisms of the fibre functor $\eta\co \FD\Rep_{\hat{A}}(B) \to \FD\Vect$.
\end{lemma}
\begin{proof}
This just combines the proofs of Lemma \ref{PNlemma2} and Proposition \ref{PNprop}. The only modification is to observe that $A= C(\hat{A}, \Cx)^{\Gal(\Cx/\R)}$, and that for any irreducible representation $\rho\co B \to \End(V)$, we necessarily have $\rho|_{A}= f\id$, for some $f \in \hat{A}$. 
\end{proof}

\section{The Betti, de Rham and harmonic functors on Banach algebras}\label{derhamsn}

\subsection{The Riemann--Hilbert correspondence}

\begin{definition}
 Given a path-connected topological space $X$ with basepoint $x$ and a unital $\R$-algebra $B$,  define the Betti representation space $\oR^B_{X,x}(B)$ by
\[
\oR^B_{X,x}(B):=  \Hom_{\gp}(\pi_1(X,x), B^{\by}),
\]
where $B^{\by}$ is the multiplicative group of units in $B$.
 
Define the representation groupoid $\cR^B_X(B)$ by $ \cR^B_X(B):= [\oR^B_{X,x}(B)/B^{\by}]$, where $B^{\by}$ acts by conjugation. Note that this is independent of the choice of basepoint (being equivalent to the groupoid of  $B^{\by}$-torsors on $X$).
\end{definition}

\begin{definition}\label{DRtorsor}
Given a connected manifold $X$ with   basepoint $x$ and a Banach algebra $B$, define the de Rham groupoid $\cR^{\dR}_X(B)$ 
to be the groupoid of smooth $B^{\by}$-bundles with flat connections. Thus  $\cR^{\dR}_X(B)$ consists of  pairs $(\sT, D)$, where $\sT$ is a  right $\sA^0_X(B^{\by})$-torsor, and $D$ is a flat connection on $\sT$. 

Explicitly, write $\sA^n_X(\ad\sT):= \sT\by_{\sA^0_X(B^{\by})}\sA^n_X(B)$, 
for the adjoint action of $B^{\by}$ on $B$. 
Then a flat connection on $\sT$ is
\[
 D\co \sT \to \sA^1(\ad\sT)
\]
satisfying 
\begin{enumerate}
 \item $D$ is a $d$-connection: $D(pg)= \ad_gD(p) + g^{-1}dg $, for $g \in \sA^0_X(B^{\by})$;
\item $D$ is flat: $(\ad D)\circ D = 0$.
\end{enumerate}

Define $\oR^{\dR}_{X,x}(B)$ to be the groupoid of triples $(\sT, D,f)$,  where $(\sT, D) \in \cR^{\dR}_X(B)$ and    $f \in x^*\sT$ is a distinguished element. Since $ \oR^{\dR}_{X,x}(B)$ has no non-trivial automorphisms, we will regard it as a set-valued functor (given by its set of isomorphism classes).
\end{definition}

Note that $B^{\by}$ acts on $\oR^{\dR}_{X,x}(B)$ by changing the framing, and that  the quotient groupoid is then equivalent to $\cR^{\dR}_X(B)$.

\begin{proposition}\label{cfBettiDR}
 For any pointed connected manifold $(X,x)$, and any Banach algebra $B$, there are canonical equivalences
\[
 \cR^{\dR}_X(B) \simeq  \cR^B_X(B), \quad    \oR^{\dR}_{X,x}(B)\cong  \oR^B_{X,x}(B) 
\]
functorial in $X,x$ and $B$.
\end{proposition}
\begin{proof}
When $B$ is a finite-dimensional matrix algebra, this is \cite[5.10]{GM}. The same proof carries over to Banach algebras, noting that the argument for existence of parallel transport (\cite[\S II.3]{KN1}) holds in this generality, since $\exp(b)=\sum_{n\ge 0} b^n/n!$ converges and is invertible for all $b \in B$.
\end{proof}

\begin{remark}\label{recoveranalytic} 
The functor $\oR^{\dR}_{X,x}$ naturally extends to a functor on the category $\pro(\Ban\Alg)$ of pro-Banach algebras, by sending  any cofiltered inverse system $\{A_i\}_i$ to $\Lim_i \oR^{\dR}_{X,x}(A_i)$. 
Since the functor $\oR^B_{X,x}$ commutes with all limits, the  equivalences of Proposition \ref{cfBettiDR} then extend to pro-Banach algebras. 
The category $\pro(\Ban\Alg)$ contains all multiplicatively convex Fr\'echet algebras (since they are countable inverse limits of Banach algebras) and indeed  all complete LMC algebras via the Arens--Michael decomposition.

For any open subset $U \subset \Cx^n$, the ring $O(U):=\Gamma(U,\O_U) $ of holomorphic functions on $U$ can be realised as a pro-Banach algebra by looking at the system of sup norms on compact subspaces. Taking quotients by finitely generated ideals $I$ then gives local models $Y:=(V(I), \O_U/I)$ for complex analytic spaces, and realises $O(Y)$ as the pro-Banach algebra $O(U)/I$. 
A complex analytic morphism from $Y$ to the variety $\Hom(\pi_1(X,x), \GL_n(\Cx))$ is then just an element of $\Hom(\pi_1(X,x), \GL_nO(Y))=\oR^B_{X,x}(\Mat_nO(Y))$, for the pro-Banach algebra $\Mat_nO(Y)$ of  morphisms $Y \to \Mat_n(\Cx)$.

We can therefore recover the analytic structure of the variety $\Hom(\pi_1(X,x), \GL_n(\Cx))$ from the set-valued functor $\oR^B_{X,x}$ on Banach algebras, and hence (by  Proposition \ref{cfBettiDR}) from the set-valued functor  $\oR^{\dR}_{X,x}$.

In \cite{Sim2}, the varieties $ \oR^B_{X,x}(\Mat_n(-))$ and $\oR^{\dR}_{X,x}(\Mat_n(-))$  are  denoted by $\oR_B(X,x,n)$. and $\oR_{\DR}(X,x,n)$.
\end{remark}

\begin{lemma}\label{tensorstr0}
 For any real Banach algebras $B,C$, there is a canonical map
\[
m\co \oR^B_{X,x}(B)\by \oR^B_{X,x}(C) \to \oR^B_{X,x}(B\ten^{\pi}_{\R}C),
\]
where $\ten^{\pi}$ is the projective tensor product. This makes $\oR^B_{X,x}$ into a symmetric monoidal functor,  with unit corresponding to the trivial representation in each $\oR^B_{X,x}(B)$.
\end{lemma}
\begin{proof}
 Given representations $\rho_1\co \pi_1(X,x) \to B^{\by}$ and $\rho_2\co \pi_1(X,x) \to C^{\by}$, we obtain $\rho_1\ten \rho_2\co \pi_1(X,x) \to (B\ten C)^{\by}$. Taking completion with respect to the projective cross norm gives the required result.
\end{proof}

\begin{remark}\label{recoveranalyticG}
Given any complex affine group scheme $G$, we may use the tensor structure on $\oR^B_{X,x}$  to recover  the affine analytic variety $ \Hom(\pi_1(X,x), G(\Cx))$ in the same spirit as Remark \ref{recoveranalytic}. Explicitly, $O(G)$ is a coalgebra, so can be written as a nested union of finite-dimensional coalgebras. Therefore $O(G)^{\vee}$ is a pro-finite-dimensional algebra, and hence a pro-Banach algebra. 

Multiplication on $G$ gives us a comultiplication $\mu \co O(G)^{\vee} \to O(G\by G)^{\vee}$. For any complex analytic space $Y$,  we may then characterise  $ \Hom(\pi_1(X,x), G(Y))$ as 
\[
 \{\rho \in \oR^B_{X,x}( O(G)^{\vee}\hten O(Y))\,:\, \mu(\rho)= m(\rho, \rho) \in \oR^B_{X,x}( O(G\by G)^{ \vee} \hten O(Y))\}.
\]
\end{remark}

\subsection{Representability of the de Rham functor}

\begin{lemma}\label{freeprorep}
 Given a free group $\Gamma= F(X)$, the functor
\[
 A \mapsto \Hom_{\gp}(\Gamma, A^{\by})        
\]
on the category of real Banach algebras is pro-representable.        
\end{lemma}
\begin{proof}
Given a function $\nu\co   X \to [1, \infty)$, let $\bar{\nu}\co \Gamma \to [1, \infty)$ be the largest function subject to the conditions
\begin{enumerate}
        \item $\bar{\nu}(1)=1$;
        \item $\bar{\nu}(x)=\bar{\nu}(x^{-1})= \nu(x)$ for all $x \in X$;
        \item $\bar{\nu}(gh)\le \bar{\nu}(g)\bar{\nu}(h)$.
\end{enumerate}
Explicitly, we write any $g \in \Gamma$ as a reduced word $g= x_1^{n_1}x_2^{n_2}\ldots x_k^{n_k}$, then set $\bar{\nu}(g):= \prod_{i=1}^k \nu(x_i)^{|n_i|}$.
We now define a norm $\|-\|_{1,\nu}$ on $k[\Gamma]$ by setting
\[
 \|\sum_{\gamma \in \Gamma} \lambda_{\gamma} \gamma\|_{1,\nu}:=    \sum_{\gamma \in \Gamma} |\lambda_{\gamma}|\cdot \bar{\nu}(\gamma).    
\]
  
Now, given any representation $\rho\co \Gamma \to A^{\by}$, we may define $\nu\co X \to [1, \infty)$ by 
\[
 \nu(x):= \max\{\|\rho(x)\|, \|\rho(x^{-1})\|\};       
\]
this at least $1$ because $1= \rho(x)\rho(x^{-1})$, so $1\le \|\rho(x)\|\cdot \|\rho(x^{-1})\|$. It follows that for all $v \in k[\Gamma] $ we have $\|\rho(v)\|\le \|v\|_{1,{\nu}}$, so $\rho$ determines a map
\[
 k[\Gamma]^{\wedge_{\nu}}\to A,       
\]
where $k[\Gamma]^{\wedge_{\nu}}$ denotes the Banach algebra obtained by completing $k[\Gamma]$ with respect to the norm  $\|-\|_{1,{\nu}}$.

Next, give $[1, \infty)^X $ the structure of a poset by saying ${\nu}_1\le {\nu}_2$ provided ${\nu}_1(x) \le {\nu}_2(x)$ for all $x \in X$. This is in fact a directed set, since we can define $\max\{{\nu}_1,{\nu}_2\}$ pointwise. There is a canonical morphism
\[
 k[\Gamma]^{\wedge_{{\nu}_2}}\to k[\Gamma]^{\wedge_{{\nu}_1}}       
\]
whenever ${\nu}_1 \le{\nu}_2$, which  gives us an inverse system  $k[\Gamma]^{\an}:= \{k[\Gamma]^{\wedge_{\nu}}\}_{\nu}$ of Banach algebras, indexed by the directed set $ ([1, \infty)^X$.

Thus we have shown that 
\[
 \Hom_{\gp}(\Gamma, A^{\by}) \cong \LLim_{{\nu} \in [1, \infty)^X}    \Hom_{k\Ban\Alg} (k[\Gamma]^{\wedge_{\nu}}, A),   
\]
functorial in Banach $k$-algebras $A$. In other words, our pro-representing object is the inverse system  $k[\Gamma]^{\an}$.
\end{proof} 

\begin{example}\label{completeZ}
Take $X= \{z\}$, so $\Gamma= \Z$,  and let ${\nu}(z)=R$. Then elements of $k[\Z]^{\wedge_{\nu}} $ are $\sum_{i \in \Z} \lambda_i z^i$ such that
\[
 \sum_{i \ge 0} |\lambda_i| R^i< \infty \quad \sum_{i \le  0} |\lambda_i| R^{-i}< \infty.
\]
Thus $\Cx[\Z]^{\wedge_{R}}$ is the ring of  analytic functions converging on the annulus $R^{-1}\le |z|\le R$. Hence $\Lim_{R} \Cx[\Z]^{\wedge_{R}}$ is the ring of analytic functions on $\Cx^*$,   while
$ \Lim_{R} \R[\Z]^{\wedge_{R}}$   is the subring consisting of functions $f$ with $\overline{f(z)}= f(\bar{z})$
 
 Contrast this with  the isometric Banach completion of $\Cx[\Z]$, which just gives $\Cx[\Z]^{\wedge_{1}}= \ell^1(\Z)$. 
\end{example}

\begin{lemma}\label{fgfreerep}
 Given a finitely generated free group $\Gamma= F(X)$, the functor
\[
 A \mapsto \Hom_{\gp}(\Gamma, A^{\by})        
\]
on the category of multiplicatively convex Fr\'echet $k$-algebras is representable. 
\end{lemma}
\begin{proof}
We may embed $\N_1$ in $[1, \infty)^X$ as a subset of the constant functions. Since $X$ is finite, $\N_1$ is a cofinal subset of  $[1, \infty)^X $, giving us an isomorphism
\[
\{k[\Gamma]^{\wedge_{\nu}}\}_{{\nu} \in [1, \infty)^X } \cong   \{k[\Gamma]^{\wedge_n}\}_{n \in \N_1 }      
\]
in the category of pro-Banach $k$-algebras. Since $\N_1$ is countable, $k[\Gamma]^{\an}:= \Lim_n k[\Gamma]^{\wedge_n}$ is a Fr\'echet algebra. 

Applying the proof of Lemma \ref{freeprorep}, we have shown that 
\[
\Hom(\Gamma, A^{\by})\cong \Hom_{k\Fr\Alg}(  k[\Gamma]^{\an} ,A)       
\]
for all Banach algebras $A$. Since any  m-convex  Fr\'echet algebra $A$ can be expressed as an inverse limit $A= \Lim_i A_i$ of Banach algebras, it follows that the same isomorphism holds for all such algebras, so the functor is representable in  m-convex Fr\'echet algebras.    
\end{proof}

\begin{proposition}\label{fgrep}
     Given a finitely generated  group $\Gamma$, the functor
\[
 A \mapsto \Hom_{\gp}(\Gamma, A^{\by})        
\]
on the category of multiplicatively convex Fr\'echet $k$-algebras is representable. 
\end{proposition}   
\begin{proof}
 Choose generators $X$ for  $\Gamma$, so $\Gamma= F(X)/K$ for some normal subgroup $K$. Lemma \ref{fgfreerep} gives a  Fr\'echet $k$-algebra $k[F(X)]^{\an}$ governing representations of $F(X)$. Since
\[
 \Hom_{\gp}(\Gamma, A^{\by})= \{f \in   \Hom_{\gp}(F(X), A^{\by})\,:\, f(K)=\{1\}\},     
\]
our functor will be represented by a quotient of $k[F(X)]^{\an}$. Specifically, let $I$ be the closed ideal of $k[F(X)]^{\an} $ generated by $\{k-1\,:\, k \in K\}$, and set $k[\Gamma]^{\an}:= k[F(X)]^{\an}/I$. This is a m-convex Fr\'echet algebra, and 
\[
 \Hom_{\gp}(\Gamma, A^{\by})\cong   \Hom_{k\Fr\Alg}(  k[\Gamma]^{\an} ,A)      
\]
for all such algebras $A$.
     
For an explicit description of $k[\Gamma]^{\an}$, note that the system of norms is given by 
\[
 \|\sum \lambda_{\gamma}\gamma \|_{1,n}= \sum |\lambda_{\gamma}| \cdot n^{w(\gamma)},       
\]
where $w(\gamma)$ is the minimal word length of $\gamma$ in terms of $X$.
\end{proof}

When combined with its tensor structure, this implies that the functor of Proposition \ref{fgrep} is a very strong invariant indeed:
\begin{lemma}\label{recovergroup}
The tensor structure of Lemma \ref{tensorstr0} gives $k[\Gamma]^{\an} $ the structure of a m-convex Fr\'echet bialgebra. The group 
\[
 G(k[\Gamma]^{\an})=\{a \in k[\Gamma]^{\an}\,:\, \mu(a)= a\ten a \in k[\Gamma]^{\an}\ten^{\pi} k[\Gamma]^{\an},\, \vareps(a)=1 \in k\}
\]
of grouplike elements of $ k[\Gamma]^{\an} $ is then $\Gamma$.
\end{lemma}
\begin{proof}
Applying the map $m$ of Lemma \ref{tensorstr0} to $(\xi,\xi)$, for the canonical element  $\xi \in\Hom_{\gp}(\Gamma,k[\Gamma]^{\an} ) $  gives us a comultiplication $\mu\co k[\Gamma]^{\an}\to k[\Gamma]^{\an}\ten^{\pi}k[\Gamma]^{\an}= k[\Gamma\by \Gamma]^{\an}$ and a co-unit $\vareps\co k[\Gamma]^{\an} \to k$. On the topological basis $\Gamma$, we must have $\mu(\gamma)= (\gamma, \gamma)$ and $\vareps(\gamma)=1$.

Expressing $a\in G(k[\Gamma]^{\an})$  as $\sum_{\gamma \in \Gamma} a_{\gamma} \gamma$,  note that the conditions become $a_{\gamma} a_{\delta}=0$ for $\gamma\ne \delta$, and $\sum a_{\gamma}=1$; thus $a=\gamma$ for some $\gamma\in \Gamma$.
\end{proof}

\begin{example}\label{completeab}
 Arguing as in example \ref{completeZ}, for $\Gamma$ abelian and finitely generated, $\Cx[\Gamma]^{\an}$ is isomorphic to the  ring of complex analytic functions on $\Hom_{\gp}(\Gamma, \Cx^*)$, while $\R[\Gamma]^{\an} \subset \Cx[\Gamma]^{\an}$ consists of $\Gal(\Cx/\R)$-equivariant functions. The multiplicative analytic functions are of course just $\Gamma$ itself.
\end{example}

Proposition \ref{cfBettiDR} then implies:
\begin{corollary}
 The functors $\oR^{\dR}_{X,x}$  and $\oR^B_{X,x}$ on real multiplicatively convex Fr\'echet algebras are representable. 
\end{corollary}

\begin{remark}\label{htpydgrk}
Adapting the ideas of \cite{htpy}, the functor $\oR^B_{X,x}$ has a natural extension to those simplicial Banach algebras $B_{\bt}$ for which $B_n \to \pi_0B$ is a nilpotent extension for each $n$. Explicitly, we could set $ \oR^B_{X,x}(B)$ to be the set of homotopy classes of maps $G(\Sing(X,x)) \to B_{\bt}^{\by}$ of simplicial groups, where $G$ is Kan's loop group. This functor  admits a  tensor  structure extending Lemma \ref{tensorstr0},

The functor $\oR^{\dR}_{X,x}$ has a natural extension to those differential graded Banach algebras
 $B_{\bt}$ for which $B_0 \to \H_0B$ is a nilpotent extension. Explicitly, $\oR^{\dR}_{X,x}(B) $ would consist of pairs $(\sT_0, D)$, where  $\sT_0$ is a $\sA^0_X(B^{\by}_0)$-torsor and $D\co \cT_0 \to \prod_n \sA^{n+1}_X\ten_{\sA^0_X}\ad \sT_n(n+1)$ is a flat hyperconnection, where
 $\ad \sT_n := \sT\by_{\sA^0_X(B^{\by}_0) }\sA^0_X(B_n)$. 

It then seems likely that \cite[Corollary \ref{htpy-bigequiv}]{htpy} should adapt to give natural isomorphisms $\oR^B_{X,x}(B) \cong \oR^{\dR}_{X,x}(NB)$, where $N$ is Dold--Kan normalisation.   
\end{remark}

\subsection{The pluriharmonic functor}

Fix a compact connected K\"ahler manifold $X$, with basepoint $x \in X$.

\begin{definition}\label{Adef}
Given a real Banach space $B$, denote the sheaf  of $B$-valued $\C^{\infty}$ $n$-forms on $X$ by $\sA^n_X(B)$, and let  $\sA_X^{\bt}$ be the resulting complex. Write $A^{\bt}(X,B):= \Gamma(X, \sA_X^{\bt}(B))$. We also write $\sA^{\bt}_X:= \sA^{\bt}_X(\R)$ and $A^{\bt}(X):= A^{\bt}(X,\R)$.
\end{definition}

\begin{definition}\label{Sdef}
  Define $S$ to be the real algebraic group  $\prod_{\Cx/\R} \bG_m$ obtained as in \cite{Hodge2} 2.1.2 from $\bG_{m,\Cx}$ by restriction of scalars. Note that there is a canonical inclusion $\bG_m \into S$.     
\end{definition}

The following is a slight generalisation of \cite[Definition \ref{mhs2-dmd}]{mhs2}:
\begin{definition}\label{dmd}
For any real Banach space $B$, there is an action of $S$ on $\sA^*_X(B)$, which we will denote by $a \mapsto \lambda \dmd a$, for $\lambda \in \Cx^* = S(\R)$. For $a \in (A^*(X)\ten \Cx)^{pq}$, it is given by
$$
\lambda \dmd a := \lambda^p\bar{\lambda}^qa.
$$
\end{definition}

\begin{definition}\label{pluritorsor}
Given a real $C^*$-algebra $B$, define $\cR^J_X(B)$ to be the groupoid of pairs $(U(\sP), D)$, where $U(\sP)$ is a  right $\sA^0_X(U(B))$-torsor, and $D$ is a pluriharmonic connection on $U(\sP)$. 

Explicitly, write $\sP:= U(\sP)\by_{\sA^0_X(U(B))}\sA^0_X(B^{\by})$, and   
\begin{align*}
\ad\sP &:= \sP\by_{\sA^0_X(B^{\by})}\sA^0_X(B)\\
&= U(\sP)\by_{\sA^0_X(U(B))}\sA^0_X(B)\\
&= [U(\sP)\by_{\sA^0_X(U(B))}\sA^0_X(\fu(B))]\oplus [U(\sP)\by_{\sA^0_X(U(B))}\sA^0_X(S(B))],
\end{align*}
 where $U(B)$ and  $B^{\by}$ act on $B$ by the adjoint action. 
Then a pluriharmonic connection on $\sP$ is
\[
 D\co U(\sP) \to \ad\sP
\]
satisfying 
\begin{enumerate}
 \item $D$ is a $d$-connection: $D(pu)= \ad_uD(p) + u^{-1}du $, for $u \in \sA^0_X(U(B))$;
\item $D$ is flat: $(\ad D)\circ D = 0$;
\item $D$ is pluriharmonic: $(\ad D) \circ D^c + (\ad D^c)\circ D=0$.
\end{enumerate}
Here, $D=d^+ + \vartheta$ comes from the decomposition of $\ad\sP$ into anti-self-adjoint and self-adjoint parts, and $D^c= i\dmd d^+ -i \dmd \vartheta$.

Define $\oR^J_{X,x}(B)$ to be the groupoid of triples $(U(\sP), D,f)$, where $(U(\sP), D) \in \cR^J_X(B)$ and  $f \in x^*U(\sP)$ is a distinguished element. Since $ \oR^J_{X,x}(B)$ has no non-trivial automorphisms, we will regard it as a set-valued functor (given by its set of isomorphism classes).
\end{definition}

\begin{remarks}\label{RJBanAlg}
Note that there is a natural action of $U(B)$ on $\oR^J_{X,x}(B)$, given by changing the framing. The quotient groupoid $[\oR^J_{X,x}(B)/U(B)]$ is thus equivalent to $\cR^J_X(B)$. In \cite[Lemma 7.13]{Sim2}, the set $\oR^J_{X,x}(\Mat_n(\Cx))$ is denoted by $\oR_{\DR}^J(X,x,n)$.

Also note that the definition of $\oR^J_{X,x}(B)$ can be extended to any real Banach $*$-algebra $B$. However, this will not be true of the harmonic functor of \S \ref{harmonicsn}.
\end{remarks}

\begin{example}\label{plurilocsys}
When $V$ is a real Hilbert space, the algebra $L(V)$ of bounded operators on $V$ is a real $C^*$-algebra. Then  $\cR^J_X(B)$ is equivalent to the groupoid of pluriharmonic local systems  $\vv$ in Hilbert spaces on $X$, fibrewise isometric to $V$. The connection $D\co \sA^0(\vv) \to \sA^1(\vv)$
must satisfy the pluriharmonic condition that $DD^c+D^cD=0$, for $D^c$ defined with respect to the smooth inner product $\vv \by \vv \to \sA^0_X$. Isomorphisms in $\cR^J_X(B) $ preserve the inner product.
\end{example}

\begin{definition}\label{deRhamproj}
 Define the de Rham projection
\[
 \pi_{\dR}\co \oR^J_{X,x}(B) \to \oR^{\dR}_{X,x}(B)
\]
by mapping $(U(\sP),D, f)$ to the framed flat torsor  $(\sP,D, f)= (U(\sP)\by_{\sA^0_X(U(B))}\sA^0_X(B^{\by}), D, f\by_{U(B)}B^{\by})$. 
\end{definition}

\begin{proposition}\label{pluriprorep}
 The functor $\oR^J_{X,x} \co C^*\Alg \to \Set$ is strictly pro-representable, by an object $E^J_{X,x} \in \pro( C^*\Alg)$.
\end{proposition}
\begin{proof}
 The final object in $C^*\Alg$ is $0$, and $\oR^J_{X,x}(0)$ is the one-point set, so $\oR^J_{X,x}$ preserves the final object. 

Given maps $A \to B \la C$ in $C^*\Alg$ and  $(p_A, p_B) \in \oR^J_{X,x}(A)\by_{\oR^J_{X,x}(B)}\oR^J_{X,x}(C)$, we get 
\begin{align*}
 \pi_{\dR}(p_A, p_C)\in  &\oR^{\dR}_{X,x}(A)\by_{\oR^{\dR}_{X,x}(B)}\oR^{\dR}_{X,x}(C)\\
&\cong \oR^B_{X,x}(A)\by_{\oR^B_{X,x}(B)}\oR^B_{X,x}(C)\\
&\cong\oR^B_{X,x}(A\by_BC).        
\end{align*}
Thus we have a flat torsor $(\sP, D) \in \oR^{\dR}_{X,x}(A\by_BC)$. 

It follows that  $p_A\cong (U(\sP_A), D)$ for some orthogonal form $U(\sP_A) \subset \sP_A=\sP\by_{\sA^0_X((A\by_BC)^{\by} )}\sA^0_X(A^{\by}) $, and similarly for $p_C$. Since the images of $p_A$ and $p_C$ are equal in   $\oR^{\dR}_{X,x}(B)$, there is a framed orthogonal isomorphism $\alpha\co U(\sP_A)\by_{\sA^0_X(U(A))}\sA^0_X(U(B)) \to U(\sP_C)\by_{\sA^0_X(U(C))}\sA^0_X(U(B))$, inducing the identity on $\sP_B$. 
Hence $\alpha$ must itself be the identity, so both $U(\sP_A)$ and $U(\sP_C)$ give the same unitary form $U(\sP_B)$ for $\sP_B$. It is easy to check the  pluriharmonic conditions, giving an element 
\[
(U(\sP_A)\by_{U(\sP_B)}U(\sP_C),D)\in \oR^J_{X,x}(A\by_BC )
\]
over $(p_A, p_C)$. This is essentially unique, so $\oR^J_{X,x}$ preserves fibre products, and hence finite limits.

Now, given a $C^*$-subalgebra $A \subset B$, the map $\oR^J_{X,x}(A)\to \oR^J_{X,x}(B)$ is injective. This follows because if two framed pluriharmonic bundles $\sP_1, \sP_2$ in $\oR^J_{X,x}(A)$ become isomorphic in $\oR^J_{X,x}(B)$, compatibility of framings ensures that the isomorphism $f$  maps $x^*\sP_1$ to  $x^*\sP_2$. Since $f$ is compatible with the connections, it thus gives an isomorphism $f\co \sP_1 \to \sP_2$, by considering the associated local systems.

Finally, given an inverse system $\{A_i\}_i$ of nested $C^*$-subalgebras of a $C^*$-algebra $B$ and an element of $\bigcap_i \oR^J_{X,x}(A_i)$, we have a compatible system  $\{(\sP_i,D_i, f_i)\}_i$. Set $\sP:= \Lim_i \sP_i$, with connection $D$ and framing $f$ induced by the $D_i$ and $f_i$. This defines a unique element of $\oR^J_{X,x}(\bigcap_i A_i)$, showing that 
 \[
 F(\bigcap_{i } A_i) \cong \bigcap_{i} F(A_i)
\]
Thus all the conditions of  Proposition \ref{repcstar} are satisfied, so $\oR^J_{X,x}$ is strictly pro-representable.
\end{proof}

\begin{definition}\label{tensorCstar}
Given pro-$C^*$-algebras $B,C$ over $k$, define $B\hat{\ten}_kC$ to be the maximal $k$-tensor product of  $B$ and $C$, as defined in \cite[Definition 3.1]{phillips}; this is again a pro-$C^*$-algebra. 
\end{definition}

\begin{lemma}\label{tensorstr}
 For any real pro-$C^*$-algebras $B,C$, there is a canonical map
\[
 m \co \oR^J_{X,x}(B)\by \oR^J_{X,x}(C) \to \oR^J_{X,x}(B\hat{\ten}_{\R}C),
\]
making $\oR^J_{X,x}$ into a symmetric monoidal functor, with unit corresponding to the trivial torsor in each $\oR^J_{X,x}(B)$.
\end{lemma}
\begin{proof}
 Given $(U(\sP), D, f, U(\sQ),E, \beta)$ on the left-hand side, we first form the $\sA^0_X(U(B\hat{\ten}C))$-torsor $U(\sR)$ given by $U(\sR):= (U(\sP)\by U(\sQ))_{\sA^0_X(U(B)\by U(C))}\sA^0_X(U(B\hat{\ten}C))$. We then define a connection $F$ on $U(\sR)$ determined by 
\[
 F(p,q,1)= (Dp,q)+ (p,Dq) \in \sA^1_X(\ad \sR)= (U(\sP)\by U(\sQ))_{\sA^0_X(U(B)\by U(C))}\sA^1_X(B\hat{\ten}C)
\]
for $p \in U(\sP), q \in U(\sQ)$
This is clearly flat and pluriharmonic, and the construction is also symmetric monoidal.
\end{proof}

Note that this gives $E^J_{X,x}$ the structure of a pro-$C^*$-bialgebra, with comultiplication $\mu \co E^J_{X,x} \to E^J_{X,x} \hten E^J_{X,x}$ coming from $m$, and counit $\vareps\co  E^J_{X,x} \to k$ coming from the trivial torsor.  

The following is immediate:

\begin{lemma}\label{Pfunctorial}
 For any morphism $f\co X \to Y$ of compact connected K\"ahler manifolds, there is a natural transformation
\[
 f^*\co \oR^J_{Y,fx} \to \oR^J_{X,x}
\]
of functors.
\end{lemma}

\subsection{Higgs bundles}

\begin{definition}
 Given  a complex Banach algebra $B$, write $\O_X(B)$ for the sheaf on $X$ given locally by holomorphic functions $X \to B$.
\end{definition}

\begin{definition}
For a complex Banach algebra $B$, a Higgs $B$-torsor on $X$ consists of an $\O_X(B)^{\by}$-torsor $\sT$, together with a Higgs form $\theta \in \ad \sT\ten_{\O_X}\Omega^1_X$, where $\ad \sT:= \sT\by_{\O_X(B)^{\by}, \ad}\O_X(B)$ satisfying 
\[
 \theta \wedge\theta = 0 \in  \sT\ten_{\O_X}\Omega^2_X.
\]
\end{definition}

\begin{definition}
 Let $\cR^{\Dol}_{X}(B)$ be  the groupoid  of Higgs $B$-torsors, and $\oR^{\Dol}_{X,x}(B)$ the groupoid of framed Higgs bundles $(\sT, \theta, f)$, where $f \co B^{\by}\to x^*\sT$ is a $B^{\by}$-equivariant isomorphism. Alternatively, we may think of $f$ as a distinguished element of $x^*\sT$. 

Note that $\oR^{\Dol}_{X,x}(B)$ is a discrete groupoid, so we will usually identify it with its set of isomorphism classes. Also note that there is a canonical action of $B^{\by}$ on $\oR^{\Dol}_{X,x}(B)$ given by the action on  the framings. This gives an equivalence
\[
 \cR^{\Dol}_{X}(B) \simeq [\oR^{\Dol}_{X,x}(B)/B^{\by}]
\]
 of groupoids.
\end{definition}

The following is immediate:
\begin{lemma}
 Giving a Higgs $B$-torsor $X$ is equivalent to giving an $\sA^0_X(B^{\by})$ torsor $\sQ$ equipped with a flat $\bar{\pd}$-connection, i.e. a map 
\[
 D''\co \sQ \to \ad\sQ:= \sQ\by_{\sA^0_X(B^{\by}), \ad}\sA^1_X(B)
\]
satisfying
\begin{enumerate}
 \item $D''(pg)= \ad_gD''(p) + g^{-1}\bar{\pd}g $, for $g \in \sA^0_X(B^{\by})$;
\item $(\ad D'')\circ D'' = 0$.
\end{enumerate}
\end{lemma}

\begin{remark}\label{Dolnotprorep}
Note that unlike the Betti, de Rham, and harmonic functors, the Dolbeault functor cannot be pro-representable in general. This is for the simple reason that a left-exact scheme must be affine, but the Dolbeault moduli space is seldom so, since it contains the Picard scheme.
\end{remark}

\begin{definition}\label{Ddecomp}
 Given  $(U(\sP), D) \in \cR^J_X(B)$, decompose $d^+$ and $\vartheta$ into $(1,0)$ and $(0,1)$ types as $d^+=\pd + \bar{\pd}$ and $\vartheta= \theta+\bar{\theta}$. Now set $D'=\pd + \bar{\theta}$ and $D''= \bar{\pd}+ \theta$. Note that $D= D'+D''$ and $D^c= iD'-iD''$.
\end{definition}

\begin{definition}
For a complex  $C^*$-algebra $B$, define the 
 Dolbeault projection map $\pi_{\Dol}\co\cR^J_X(B) \to \cR^{\Dol}_X(B)$ by sending $(U(\sP),D)$ to $(\sP\by_{\sA^0_X(B^{\by})}\sA^0_X(B^{\by}), D'')$.
\end{definition}

\subsection{The harmonic functor}\label{harmonicsn}

We now let $X$ be any  compact Riemannian real manifold.

\begin{definition}
 Given 
a compact Riemannian manifold $X$, a real $C^*$-algebra $B$,  a  right $\sA^0_X(U(B))$-torsor $U(\sP)$ and a flat connection  
\[
 D\co U(\sP) \to \ad\sP,
\]
 say that $D$ is a harmonic connection if $(d^+)^*\vartheta=0 \in \Gamma(X,\ad\sP)$, for $d^+, \vartheta$ defined as in Definition \ref{pluritorsor}, and the adjoint $*$ given by combining the involution $*$ on $\ad\sP$ with the adjoint on $\sA^*_X$ given by the K\"ahler form.
\end{definition}

\begin{lemma}
 A flat connection $D$ as above on a compact K\"ahler manifold is harmonic if and only if it is pluriharmonic.
\end{lemma}
\begin{proof}
 The proof of \cite[Lemma 1.1]{Simpson} carries over to this generality.
\end{proof}

\begin{definition}
The lemma allows us to  extend Definitions \ref{pluritorsor},\ref{deRhamproj} to any compact Riemannian manifold $X$, replacing pluriharmonic with harmonic in the definition of $\cR^J_X(B)$ and $\oR^J_{X,x}(B)$.
\end{definition}

\begin{proposition}\label{harmprorep}
 The functor $\oR^J_{X,x} \co C^*\Alg_{\R} \to \Set$ is strictly pro-representable, by an object $E^J_{X,x} \in \pro( C^*\Alg{\R})$.
\end{proposition}
\begin{proof}
 The proof of Proposition \ref{pluriprorep} carries over.
\end{proof}

Note that Lemma \ref{tensorstr} carries over to the functor $\oR^J_{X,x}$ for any compact Riemannian manifold $X$.

The following is immediate:

\begin{lemma}\label{Pfunctorial2}
 For any local isometry $f\co X \to Y$ of compact connected real Riemannian manifolds, there is a natural transformation
\[
 f^*\co \oR^J_{Y,fx} \to \oR^J_{X,x}
\]
of functors.
\end{lemma}

Note that this is much weaker than Lemma \ref{Pfunctorial}, the pluriharmonic functor being \emph{a priori} functorial with respect to all morphisms.

\section{Analytic non-abelian Hodge theorems}\label{cfansn} 

\subsection{The de Rham projection}\label{dRprojsn}

 Fix a compact connected real Riemannian  manifold $X$, with basepoint $x \in X$.

The argument of \cite[Lemma 7.17]{Sim2} (which is only stated for $X$ K\"ahler) shows that $\pi_{\dR}$ gives a homeomorphism
\[
\oR^J_{X,x}(\Mat_n(\Cx))/U(n) \to  \Hom(\pi_1(X,x), \GL_n(\Cx))//\GL_n(\Cx),
\]
where $//$ denotes the coarse quotient (in this case, the Hausdorff completion of the topological quotient).

As an immediate consequence, note that
\[
 \oR^J_{X,x}(\Cx) \to  \Hom_{\gp}(\pi_1(X,x), \Cx^*)
\]
is a homeomorphism. Thus the abelianisation of $E^J_{X,x}\ten \Cx$ is isomorphic to the commutative $C^*$-algebra $C( \Hom(\pi_1(X,x), \Cx^*), \Cx)$, with $\oR^J_{X,x} \subset \oR^J_{X,x}\ten \Cx$ consisting of $\Gal(\Cx/\R)$-equivariant functions.

We now adapt these results to recover a finer comparison between the respective functors.

\subsubsection{Harmonic representations}

\begin{proposition}\label{uniquemetric}
For a compact Riemannian manifold $X$ and all real $C^*$-algebras $B$, the de Rham projection 
\[
 \pi_{\dR}\co \oR^J_{X,x}(B) \to \oR^{\dR}_{X,x}(B)
\]
has the property that if $p_1, p_2 \in  \oR^J_{X,x}(B)$ and if $\ad_b\pi_{\dR}(p_1)= \pi_{\dR}(p_2)$ for some strictly positive self-adjoint element $b \in B$, then $p_1=p_2$.

Thus
\[
\pi_{\dR}\co \oR^J_{X,x}(B)/U(B) \to  \oR^{\dR}_{X,x}(B)/B^{\by}
\]
 is injective.
\end{proposition}
\begin{proof}
The  first statement above implies the second: it suffices to show that for any $(U(\sP), D, f) \in  \oR^J_{X,x}(B)$,  there are no other harmonic representations in the $B_{++}$-orbit of $\pi_{\dR}(U(\sP),D, f)$. Since $B_{++}= \exp(S(B))$ (by the continuous functional calculus), we can equivalently look at the orbit under the exponential action of the set $S(B)$.

We adapt the proof of \cite[Proposition 2.3]{corlette}. The harmonic condition $(d^+)^*\vartheta=0$ is equivalent to saying that for all $\xi \in A^0(X, i\ad \sP)$
\[
 \<\vartheta, d^+\xi\>=0 \in A^0(X, B),
\]
where $\<-,-\>$ is defined using the Riemannian metric. 

Now, the set of flat connections on $\sP$ admits a  gauge action $\star$ of the smooth automorphism group of $\sP$, and hence via exponentiation an action of the additive group $\Gamma(X,\ad \sP)$. An isomorphism in $B_{++}$ between two flat connections corresponds to an element of  $\exp(\Gamma(X, S(\ad \sP)))$, for $S(\ad\sP) \subset \ad \sP$ consisting of symmetric elements, giving a gauge between the respective connections. Thus the $S(B)$-orbit above is given by looking at the orbit of $D$ under $\Gamma(X, S(\ad \sP))$.

By analogy with \cite[Proposition 2.3]{corlette}, we fix $\xi \in A^0(X, S(\ad \sP))$, let $d^+_t, \vartheta_t$ be the anti-self-adjoint and self-adjoint parts of  $\exp(\xi t)\star D$, and set
\[
 f(t):= \< \vartheta_t, \vartheta_t \> \in A^0(X,B),
\]
Now, $\frac{d}{dt}(\exp(\xi t)\star D) =(\exp(\xi t)\star D)\xi = d^+_t\xi + \vartheta_t\wedge \xi$, so  $\frac{d}{dt}\vartheta_t= d^+_t\xi$ and 
\[
 f'(t)= 2\<\xi, (d^+_t)^*\vartheta_t \> \in A^0(X,B).
\]
In other words, $D_t$ is harmonic if and only if $f'(t)=0$ for all $\xi$.

Now, if we set $\hat{D}_t:= d^+_t-\vartheta_t$,  the calculations of \cite[Proposition 2.3]{corlette} adapt to give 
\[
 2f''(t)= \|D_t\xi +\hat{D}_{t}\xi\|^2+ \|D_t\xi -\hat{D}_{t}\xi\|^2\in A^0(X,B),
\]
where $\|v\|^2:= \<v,v\>$; unlike Corlette, we are only taking inner product with respect to the K\"ahler metric, not imposing an additional inner product on $B$.

Note that $f''(t)$  an element of $A^0(X,B_{+})$, which lies in $A^0(X,B_{++})$ unless $D_t\xi=0$. If we start with a harmonic connection $D$, this implies that $\exp(\xi)\star D$ is harmonic if and only if $D\xi=0$. However, when $D\xi=0$ we have $\exp(\xi)\star D=D$, showing that $D$ is the unique harmonic connection in its $B_{++}$-orbit.
\end{proof}

\begin{corollary}\label{Zpluri}
 For a complex $C^*$-algebra $B$ and an element $p \in \oR^J_{X,x}(B^{\by})$, the centraliser $\z(\pi_{\dR}(p),B^{\by})$ of $\pi_{\dR}(p)$ under the adjoint action of $B$ is given by 
\[
\z(\pi_{\dR}(p),B^{\by})= \exp(\{ b \in S(B)\,:\, e^{ibt} \in \z(p, U(B)) \forall t \in \R\})\rtimes\z(p, U(B));
\]
beware that this is the semidirect product of a set with a group.
\end{corollary}
\begin{proof}
Take $g \in \z(\pi_{\dR}(p),B^{\by})$, and observe that the polar decomposition allows us to write $g = \exp(b)u$, for $u\in U(B)$ and $b \in S(B)$. Since $\pi_{\dR}$ is $U(B)$-equivariant, we have
\[
 \ad_{\exp(b)}(\pi_{\dR}(\ad_u(p)))= \ad_g(\pi_{\dR}(p))=\pi_{\dR}(p).
\]
Thus Proposition \ref{uniquemetric} implies that $ \ad_u(p)=p$, so $u \in \z(p, U(B))$.

Since $\z(\pi_{\dR}(p),B^{\by})$ is a group, this implies that $\exp(b) \in \z(\pi_{\dR}(p),B^{\by})$, and hence that $\exp(b)$ commutes with the image of $\pi_{\dR}(p)$. We may  apply the continuous functional calculus to take logarithms, showing that $b$ itself commutes with the image of $\pi_{\dR}(p)$, so $ib t$ does also. But then $\exp(ib t)\in \z(p, U(B))$ for all $t$, as required.

Conversely, if $\exp(ib t)\in \z(p, U(B))$ for all $t$, then $\exp(-ib t)\pi_{\dR}(p)\exp(ib t)= \pi_{\dR}(p)$, and differentiating in $t$ for each element of $\pi_1(X,x)$, we see that $ib$ commutes with $\pi_{\dR}(p)$. Thus $\exp(b) \in  \z(\pi_{\dR}(p),B^{\by})$.
\end{proof}

\subsubsection{Topological representation spaces and completely bounded maps}

\begin{lemma}\label{Banachplurifunctor}
 For the real pro-$C^*$-algebra $E^J_{X,x}$ of Proposition \ref{pluriprorep},  there is a canonical map $\pi_{\dR}\co\Hom_{\pro(\Ban\Alg)}(E^J_{X,x},B) \to \oR^{\dR}_{X,x}(B)$, functorial in real Banach algebras $B$.
\end{lemma}
\begin{proof}
 Given $f\co E^J_{X,x} \to B$, Lemma \ref{Cstarimage} factors $f$ as the composition of  a surjective $C^*$-homomorphism $g\co E^J_{X,x}\to C$ and a continuous embedding $C \into B$. The de Rham projection of Definition \ref{deRhamproj} then gives us an element  $\pi_{\dR}(g)\in \oR^{\dR}_{X,x}(C)$. Combining this with the  embedding $C^{\by} \to B^{\by}$ then provides the required element of  $ \oR^{\dR}_{X,x}(B)$.
\end{proof}

\begin{proposition}\label{banachreps}
 For  any real $C^*$-algebra $B$, the map of Lemma \ref{Banachplurifunctor} induces an injection
\[
\pi_{\dR}\co \Hom(E^J_{X,x},B)_{cb} \into \oR^{\dR}_{X,x}(B),
\]
for the completely bounded morphisms of Definition \ref{cbdef}.
\end{proposition}
\begin{proof}
Since $B$ can be embedded as a closed $C^*$-subalgebra of $L(H)$ for some complex Hilbert space $H$, we may replace $B$ with $L(H)$. By  Lemma \ref{KSPlemma}, any completely bounded homomorphism $f\co E^J_{X,x} \to L(H)$ is conjugate to a $*$-morphism, since $E^J_{X,x}$ is a pro-$C^*$-algebra. Therefore Proposition \ref{uniquemetric} shows that 
\[
 \Hom(E^J_{X,x},L(H))_{cb}/\GL(H) \into \oR^{\dR}_{X,x}(L(H))/\GL(H).
\]

Take a homomorphism  $f\co E^J_{X,x} \to L(H)$ of $C^*$-algebras;  it suffices to show that the centraliser of $f$ and of $\pi_{\dR}(f)$ are equal. By Corollary \ref{Zpluri}, we know that
\[
\z(\pi_{\dR}(f),\GL(H))= \exp(\{ b \in S(L(H))\,:\, e^{ibt} \in \z(f, U(H)) \forall t \in \R\})\rtimes\z(f, U(H)).
\]

If $e^{ibt}$ commutes with $f$ for all $f$, then $e^{ibt}fe^{-ibt}=f$, and differentiating in $t$ shows that $b$ commutes with $f$. Therefore $\exp(b)$ commutes with $f$, showing that
\[
 \z(\pi_{\dR}(f),\GL(H))\subset \z(f,\GL(H)).
\]
The reverse inclusion is automatic, giving the required result.
\end{proof}

\begin{remark}\label{cfredpi}
In \cite{Simpson, mhs2}, the pro-reductive fundamental group  $\pi_1(X,x)^{\red}_k$ is studied --- this is an affine group scheme over $k$. By Tannakian duality (\cite[Ch. II]{tannaka}), we can interpret the dual $O(\pi_1(X,x)^{\red}_{\R})^{\vee}$ of the ring of functions as the ring of discontinuous  $\Gal(\Cx/\R)$-equivariant endomorphisms  of $\eta_x^{\dR, \ss}$. 

The group scheme $\pi_1(X,x)^{\red}_k$ encodes all the information about the sets of finite-dimensional representations of $\pi_1(X,x)$. As we will now see, $ (E^J_{X,x})_{\PN}$ encodes all the information about their topologies as well.
\end{remark}

\begin{theorem}\label{Hsstopthm}
For any positive integer $n$, $\pi_{\dR}$ gives a homeomorphism $\pi_{\dR,\ss}$ between the space 
$\Hom_{\pro(\Ban\Alg)}(E^J_{X,x},\Mat_n(\Cx))$ with the topology of pointwise convergence, and the subspace of $\oR^{\dR}_{X,x}(\Mat_n(\Cx))$ whose points correspond to semisimple local systems.      
\end{theorem}
\begin{proof}
The isomorphism $\pi_{\dR,\ss}$ is given on points by the proof of \cite[Theorem 3.3]{corlette}, since  completely bounded algebra homomorphisms $E^J_{X,x} \to \Mat_n(\Cx)$ are those conjugate to $*$-homomorphisms, which in turn correspond to harmonic local systems. We need to show that this is a homeomorphism. 

Consider the map $\pi_{\dR}^{\sharp}\co E^{\dR}_{X,x}\to  E^J_{X,x}$ of pro-Banach algebras. If $T_i \to T$ is a convergent net in $\Hom_{\pro(\Ban\Alg)}(E^J_{X,x},\Mat_n(\Cx))$, then $T_i(\pi_{\dR}^{\sharp}(\gamma)) \to T(\pi_{\dR}^{\sharp}(\gamma))$, so  $\pi_{\dR,\ss}$ is continuous.

Now, the de Rham projection $\pi_{\dR} \co \oR^J_{X,x}(B)\to \oR^{\dR}_{X,x}(B)$ is automatically injective for $C^*$-algebras $B$. Thus the inclusion in $E^J_{X,x}$ of the pro-$C^*$-subalgebra  generated by $\pi_{\dR}^{\sharp}(\pi_1(X,x))$ must be an epimorphism, since $\R[\pi_1(X,x)]$ is dense in $E^{\dR}_{X,x}$. By \cite[Proposition 2]{reidEpimorphisms}, an epimorphism of $C^*$-algebras is surjective, so  $E^J_{X,x}$ must be generated as a pro-$C^*$-algebra by $\pi_{\dR}^{\sharp}(\pi_1(X,x))$, and as a pro-Banach algebra by $\pi_{\dR}^{\sharp}(\pi_1(X,x)) \cup \pi_{\dR}^{\sharp}(\pi_1(X,x))^*$.

Now, if we have a convergent sequence $\pi_{\dR}(T_i) \to \pi_{\dR}(T)$, then $T_i(\pi_{\dR}^{\sharp}\gamma) \to T(\pi_{\dR}^{\sharp}\gamma)$ for all $\gamma \in \pi_1(X,x)$, so it suffices to show that the same holds for $(\pi_{\dR}^{\sharp}\gamma)^* $. Given $T \in \Hom_{\pro(\Ban\Alg)}(E^J_{X,x},\Mat_n(\Cx))$, define $T^*$ by $T^*(e):= T(e^*)^*$. We wish to show that $T_i^*(\pi_{\dR}^{\sharp}\gamma) \to T^*(\pi_{\dR}^{\sharp}\gamma) $, which will follow if $\pi_{\dR}(T_i^*) \to \pi_{\dR}(T^*)$.
 
As in the proof of Proposition \ref{banachreps}, we can write $T= \ad_g(S)$, for $S \co E^J_{X,x},\Mat_n(\Cx)$ a $*$-homomorphism and $g \in \GL_n(\Cx)$. Then $T^*= \ad_{(g^*)^{-1}}(S)= \ad_{(gg^*)^{-1}}(T)$; this means that if we write $\pi_{\dR}(T)= (\sV,D, f)$, then $\pi_{\dR}(T^*)=(\sV,D, (f^*)^{-1}) $, where $f^*\co   x^*\sV\to\Cx^n$ is defined using the harmonic metric on $\sV$ and the standard inner product on $\Cx^n$.

Explicitly, this means that we can describe the involution $*$ on semisimple elements of $\oR^{B}_{X,x}(\Mat_n(\Cx))$ by $\rho^*= (C(\rho)^{-1})^{\dagger}$, where $C$ is the Cartan involution of \cite{Simpson}. If $D= d^++\vartheta$ is the decomposition  into anti-hermitian and hermitian parts with respect to the harmonic metric, then  $C(\sV,d^++\vartheta , f)=(\sV,d^+-\vartheta , f)$. The proof of \cite[Theorem 3.3]{corlette} ensures that the decomposition $D \mapsto (d^+, \vartheta)$ is continuous in $D$, so $C$ is continuous. 
 Hence $\pi_{\dR}(T) \mapsto \pi_{\dR}(T^*)$ is also continuous, which gives the convergence required. 
\end{proof}

\subsubsection{The polynormal completion and Tannaka duality}\label{recovertopology}

\begin{definition}\label{FDDRcat}
 Let $\FD\oR^{\dR}_{X,x}$ be the category of pairs $(V, p)$ for $V\in \FD\Vect$ and  $p  \in \oR^{\dR}_{X,x}(\End(V))$.  Morphisms $f\co (V_1, p_1)\to (V_2, p_2)$ are given by linear maps $f\co V_1 \to V_2$  for which the adjoint action of 
\[
 \begin{pmatrix}
  \id & 0 \\ f & \id      
 \end{pmatrix}
  \in
\begin{pmatrix}
  \End(V_1) & 0 \\ \Hom(V_1,V_2) & \End(V_2)      
 \end{pmatrix}     
\]
on $FD\oR^{\dR}_{X,x} \left(\begin{smallmatrix}
  \End(V_1) & 0 \\ \Hom(V_1,V_2) & \End(V_2)      
 \end{smallmatrix}\right)$
fixes $p_1 \oplus p_2$.

Write $\eta_x^{\dR}\co \FD\oR^{\dR}_{X,x} \to \FD\Vect$ for the fibre functor $(V,p) \mapsto V$. Let $\FD\oR^{\dR,\ss}_{X,x} \subset \FD\oR^{\dR}_{X,x}$  be the full subcategory in which objects correspond to semisimple local systems, with fibre functor $\eta_x^{\dR, \ss}$.  
\end{definition}

\begin{proposition}\label{PNssprop}
The ring $ (E^J_{X,x})_{\PN}$ is isomorphic to the ring of continuous  $\Gal(\Cx/\R)$-equivariant  endomorphisms of $\eta_x^{\dR, \ss}$.
\end{proposition}
\begin{proof}
This just combines Lemma \ref{PNlemma2} and
Theorem \ref{Hsstopthm}.
\end{proof}

\begin{remark}\label{nonss}
This leads us to contemplate the structure of the ring of   continuous  endomorphisms $f$ of $\eta_x^{\dR}$. Any finite-dimensional $\Cx$-algebra arises as a subalgebra of some matrix algebra, so any such $f$ induces continuous maps $\Hom_{\gp}(\pi_1(X,x),B^{\by}) \to B$ for all finite-dimensional algebras $B$. In particular, this holds when $B= \Mat_n(A)$ for some Artinian $\Cx$-algebra $A$, from which it follows that the maps  
\[
 f_V \co \oR^{\dR}_{X,x}(\End(V)) \to \End(V)
\]
are all analytic. In other words, any continuous  endomorphism  of $\eta_x^{\dR}$ is automatically analytic.

When $\pi_1(X,x)$ is abelian, this ensures that the ring $(E^{\dR}_{X,x})_{\FD}\ten\Cx$ of such endomorphisms is the ring of complex analytic functions on $\Hom_{\gp}(\pi_1(X,x), \Cx^*)$, which by Example \ref{completeab} is just $\Cx[\pi_1(X,x)]^{\an}$. 

In general, the ring $(E^{\dR}_{X,x})_{\FD}$ is an inverse limit of polynormal Banach algebras, but it is not clear to the author whether it is the pro-polynormal completion of the Fr\'echet algebra $E^{\dR}_{X,x}$. In general, the map $\pi_1(X,x)\to (E^{\dR}_{X,x})_{\FD}$ need not be injective: \cite[\S 6.5]{Am} gives examples of 
K\"ahler groups with 
no faithful linear representations.
\end{remark}

\begin{definition}\label{ssRdef}
 Given a    $k$-normal real $C^*$-algebra $B$, define $ \oR^{\dR,\ss}_{X,x}(B)\subset \oR^{\dR}_{X,x}(B)$  to be the subspace  consisting of those $p$ for which $\psi(p)$ corresponds to a semisimple local system for all $\psi \co B \to \Mat_k(\Cx)$.  
\end{definition}

\begin{corollary}\label{PNrepss}
For any $k$-normal real $C^*$-algebra $B$,   $ \oR^{\dR,\ss}_{X,x}(B)$ is isomorphic to the set of continuous algebra homomorphisms $E^J_{X,x} \to B$.
\end{corollary}
\begin{proof}
Since $B$ is $k$-normal, any such morphism $ E^J_{X,x} \to B$ factors uniquely through $ (E^J_{X,x})_{\PN}$
 By Proposition  \ref{PNssprop}, a homomorphism $(E^J_{X,x})_{\PN} \to B $  corresponds to a continuous  $\Gal(\Cx/\R)$-equivariant functor $p^* \co \FD\Rep(B) \to \FD\oR^{\dR,\ss}_{X,x}$ of topological categories fibred over $\FD\Vect$. An element $p \in \oR^{\dR}_{X,x}(B)$ satisfies this condition provided $p^*$ maps to $\FD\oR^{\dR,\ss}_{X,x} \subset  \FD\oR^{\dR}_{X,x}$.       
\end{proof}

\begin{remark}\label{likelygenerality}
It is natural to ask whether the non-abelian Hodge theorem of \cite{Sim2} extends from finite-dimensional matrix algebras to more general $C^*$-algebras $B$. Proposition \ref{PNssprop} can be thought of as an extension of the correspondence to polynormal $C^*$-algebras, but it seems unlikely to adapt much further, because the arguments of \cite{Sim2,corlette} rely on sequential compactness of $U_n$.  However, it is conceivable that suitable solutions of the heat equation might give rise to asymptotic $C^*$-homomorphisms.       
\end{remark}


\subsection{Residually finite-dimensional completion, products and complex tori}\label{RFDsn}

\begin{definition}\label{RFDdef}
A pro-$C^*$-algebra $A$ is said to be \emph{residually finite-dimensional} if it has a separating family of finite-dimensional $*$-representations.

Given a pro-$C^*$-algebra $A$, define the pro-$C^*$-algebra $A_{\RFD}$ to be the universal residually finite-dimensional quotient of $A$. Explicitly, $A_{\RFD}$ is the quotient of  $A$ with respect to the pro-ideal given by the system of kernels of finite-dimensional $*$-representations of $A$.
\end{definition}

Note that polynormal $C^*$-algebras are residually finite-dimensional, so we have completions $A \to A_{\RFD} \to A_{\PN}$ for general $A$.

\begin{proposition}\label{productprop}
Given  compact connected K\"ahler   manifolds $X$ and $Y$, there is an isomorphism
$(E^J_{X\by Y, (x,y)})_{\RFD} \cong (E^J_{X,x})_{\RFD}\hat{\ten}( E^J_{Y,y})_{\RFD}$. 
\end{proposition}
\begin{proof}
The projections give canonical elements of $\oR^J_{X\by Y, (x,y)}(E^J_{X,x})$ and $\oR^J_{X\by Y, (x,y)}(E^J_{Y,y})$, which by Lemma \ref{tensorstr} give rise to  
 a canonical map
\[
 f\co E^J_{X\by Y, (x,y)} \to E^J_{X,x}\hat{\ten} E^J_{Y,y}.
\]

By \cite{corlette}, every finite-dimensional representation of $E^J_{X \by Y, (x,y)}$ corresponds to a semisimple representation of $\pi_1(X\by Y,(x,y))= \pi_1(X,x) \by \pi_1(Y,y)$, so  factors through $E^J_X\hten E^J_Y$. Since $(E^J_{X \by Y, (x,y)})_{\RFD} \subset \prod_i \End(V_i)$ where $V_i$ ranges over finite-dimensional irreducible representations, this implies that 
\[
 f_{\RFD}\co  (E^J_{X\by Y, (x,y)})_{\RFD} \to (E^J_{X,x})_{\RFD}\hat{\ten}( E^J_{Y,y})_{\RFD}
\]
is injective. However, the basepoint $y$ gives us a map $X \to X \by Y$, and hence  $E^J_{X,x} \to E^J_{X \by Y, (x,y)} $, ensuring that $(E^J_X)_{\RFD} $ lies in the image of $f_{\RFD}$; a similar argument applies to $Y$. Thus $f_{\RFD} $ is surjective, and hence an isomorphism.  
\end{proof}

\begin{remark}
 Note that we have only imposed the hypothesis that the manifolds be K\"ahler in order to use the functoriality properties of Lemma \ref{Pfunctorial}, since  Lemma \ref{Pfunctorial2} is too weak to apply to the maps between $X\by Y$ and $X$.
\end{remark}

\begin{lemma}\label{ablemma}
 Given a compact connected K\"ahler   manifold $X$, the commutative quotient $(E^J_{X,x})^{\ab} $ is given by
\begin{align*}
(E^J_{X,x})^{\ab}\ten \Cx &= C(\H^1(X,\Cx^*),\Cx)\\
(E^J_{X,x})^{\ab} &= \{f \in C(\H^1(X, \Cx^*),\Cx)\,:\, f(\bar{\rho})= \overline{f(\rho)}\}. \\
\end{align*}
\end{lemma}
\begin{proof}
 Since $(E^J_{X,x})^{\ab} $ is a commutative $C^*$-algebra, the Gelfand--Naimark theorem gives $(E^J_{X,x})^{\ab}\ten \Cx\cong C(\Hom(E^J_{X,x}, \Cx),\Cx) $, with $E^J_{X,x})^{\ab}$ given by $\Gal(\Cx/\R)$-invariants. By \S \ref{recovertopology},  we have a homeomorphism $\Hom(E^J_{X,e}, \Cx) \cong \H^1(X,\Cx^*)$, which completes the proof.

\end{proof}

\begin{corollary}\label{abcor}
For $X$ a complex torus with identity $e$ and a fixed Riemannian metric, we have
\begin{align*}
(E^J_{X,e})_{\RFD}\ten \Cx &= C(\H^1(X,\Cx^*),\Cx)\\
(E^J_{X,e})_{\RFD} &= \{f \in C(\H^1(X, \Cx^*),\Cx)\,:\, f(\bar{\rho})= \overline{f(\rho)}\}. \\
\end{align*}
 \end{corollary}
\begin{proof}
Multiplication on $X$ gives a pointed morphism $X \by X \to X$ and hence by functoriality of $P$ in $X$ and Proposition \ref{productprop}, we have a morphism
\[
( E^J_{X,e})_{\RFD}\hat{\ten}(E^J_{X,e})_{\RFD} \to (E^J_{X,e})_{\RFD}
\]
of real pro-$C^*$-algebras, and we may apply Lemma \ref{ablemma}. 
\end{proof}

\begin{remark}
If it were the case that all irreducible representations of $\pi_1(X,x)$ were harmonic and similarly for $\pi_1(Y,y)$, then the  proof of Proposition \ref{productprop} would adapt to show that  $E^J_{X\by Y, (x,y)} \cong E^J_{X,x}\hat{\ten} E^J_{Y,y}$. As in the proof of Corollary \ref{abcor}, that would then imply commutativity of $E^J_{X,e}$ for complex tori $(X,e)$, giving $E^J_{X,e}\ten \Cx = C(\Hom(\pi_1(X,e), \Cx^*))$.
\end{remark}

\begin{lemma}\label{EJabgp}
 Given a compact connected K\"ahler   manifold $X$, the grouplike elements $G((E^J_{X,x})^{\ab}) $ (see Lemma \ref{recovergroup}) of the commutative quotient $(E^J_{X,x})^{\ab} $ are given by
\begin{align*}
 G((E^J_{X,x})^{\ab}\ten \Cx) &\cong \H_1(X,\Z\oplus\Cx) \\
 G((E^J_{X,x})^{\ab}) &\cong \H_1(X,\Z\oplus\R), 
\end{align*}
with the map $\pi_1(X,x)^{\ab} \to G((E^J_{X,x})^{\ab})$ given by the diagonal map $\Z \to \Z \oplus \R$.
\end{lemma}
\begin{proof}
 The coalgebra structure on  $(E^J_{X,x})^{\ab}$  corresponds under Lemma \ref{ablemma} to the group structure on $\H^1(X, \Cx^*) $. Thus $G((E^J_{X,x})^{\ab}) $ consists of continuous functions $f\co \H^1(X, \Cx^*)\to \Cx$ with $f(1)=1$ and $f(ab)=f(a)f(b)$. 

We have an isomorphism $\Cx^* \cong S^1 \by \R$, given by $re^{i\phi} \mapsto (\phi, r)$. Thus $\H^1(X,\Cx^*)\cong \H^1(X,S^1\by\H^1(X,\R)$. By Pontrjagin duality, a continuous group homomorphism $\H^1(X,S^1) \to \Cx^*$ is just an element of $\H_1(X,\Z)$, and a continuous group homomorphism $\H^1(X,\R) \to \Cx^*$ is  an element of $\H_1(X,\Cx)$. 
\end{proof}

\subsection{The Dolbeault projection}\label{Dolsn}

Now let $X$ be a compact connected  K\"ahler manifold with basepoint $x \in X$.

\begin{proposition}\label{uniquemetric2}
For  all complex $C^*$-algebras $B$, the Dolbeault projection
\[
 \pi_{\Dol}\co \oR^J_{X,x}(B) \to \oR^{\Dol}_{X,x}(B)
\]
has the property that if $p_1, p_2 \in  \oR^J_{X,x}(B)$ and if $\ad_b\pi_{\Dol}(p_1)= \pi_{\Dol}(p_2)$ for some strictly positive self-adjoint element $b \in B$, then $p_1=p_2$.

 Thus
\[
\pi_{\Dol}\co \oR^J_{X,x}(B)/U(B) \to  \oR^{\Dol}_{X,x}(B)/B^{\by}
\]
 is injective.
\end{proposition}
\begin{proof}
The proof of Proposition \ref{uniquemetric} adapts, replacing $D$ with $D''$.
\end{proof}

\begin{corollary}\label{Zpluri2}
 For an element $p \in \oR^J_{X,x}(B^{\by})$, the centraliser $\z(\pi_{\Dol}(p),B^{\by})$ of $\pi_{\Dol}(p)$ under the adjoint action of $B$ is given by 
\[
\z(\pi_{\Dol}(p),B^{\by})= \exp(\{ b \in S(B)\,:\, e^{ibt} \in \z(p, U(B)) \forall t \in \R\})\rtimes\z(p, U(B));
\]
beware that this is the semidirect product of a set with a group.
\end{corollary}
\begin{proof}
The proof of Corollary \ref{Zpluri} carries over.
\end{proof}

\begin{proposition}\label{banachreps2}
 For the real pro-$C^*$-algebra $E^J_{X,x}$ of Proposition \ref{pluriprorep},  there is a canonical map $\Hom_{\pro(\Ban\Alg)}(E^J_{X,x},B) \to \oR^{\Dol}_{X,x}(B)$, functorial in complex Banach algebras $B$. This induces an injection 
\[
 \Hom(E^J_{X,x},B)_{cb} \into \oR^{\Dol}_{X,x}(B)
\]
whenever $B$ is a $C^*$-algebra.
\end{proposition}
\begin{proof}
The proofs of Lemma \ref{Banachplurifunctor} and Proposition \ref{banachreps} carry over to this context, replacing Proposition \ref{uniquemetric} and Corollary \ref{Zpluri} with Proposition \ref{uniquemetric2} and Corollary \ref{Zpluri2}.
\end{proof}

\begin{theorem}\label{Hsttopthm}
For any positive integer $n$, there is a homeomorphism $\pi_{\Dol,\st}$ between the space 
$\Hom_{\pro(\Ban\Alg)}(E^J_{X,x},\Mat_n(\Cx))$ with the topology of pointwise convergence, and the subspace of $\oR^{\Dol}_{X,x}(\Mat_n(\Cx)) $ consisting of polystable Higgs bundles $E$ with $\ch_1(E)\cdot [\omega]^{\dim X -1}=0$ and $\ch_2(E)\cdot [\omega]^{\dim X -2}=0 $.    \end{theorem}
\begin{proof}
The isomorphism of points is given by \cite[Theorem 1]{Simpson}. Replacing Proposition \ref{banachreps} with Proposition \ref{banachreps2}, the argument from the proof of Theorem \ref{Hsstopthm} shows that the map $\pi_{\Dol}\co \Hom_{\pro(\Ban\Alg)}(E^J_{X,x},\Mat_n(\Cx))\to \oR^{\Dol}_{X,x}(\Mat_n(\Cx)) $ is continuous, so we just need to show that it is open. 

Now, \cite[Proposition 7.9]{Sim2} implies that the isomorphism $\pi_{\dR,\ss} \circ  \pi_{\Dol,\st}^{-1}$ is continuous. Since $\pi_{\dR,\ss}$ is a homeomorphism by Theorem \ref{Hsstopthm}, $\pi_{\Dol,\st}$ must also be a homeomorphism.
\end{proof}

\begin{definition}\label{FDDolcat}
 Let $\FD\oR^{\Dol}_{X,x}$ be the category of pairs $(V, p)$ for $V\in \FD\Vect$ and  $p  \in \oR^{\Dol}_{X,x}(\End(V))$, with morphisms defined by adapting the formulae of Definition \ref{FDDRcat}.

 Let $\FD\oR^{\Dol,\st}_{X,x} \subset \FD\oR^{\Dol}_{X,x}$  be the full subcategory in which objects correspond to those of Theorem \ref{Hsttopthm}.
Write $\eta_x^{\Dol}\co \FD\oR^{\Dol}_{X,x} \to \FD\Vect$,  $\eta_x^{\Dol,\st}\co \FD\oR^{\Dol,\st}_{X,x} \to \FD\Vect$ for the fibre functors $(V,p) \mapsto V$.
\end{definition}

\begin{proposition}\label{PNstprop}
The ring $ (E^J_{X,x})_{\PN}\ten \Cx$ is isomorphic to the ring of continuous   endomorphisms of $\eta_x^{\Dol, \st}$.
\end{proposition}
\begin{proof}
The proof of Proposition \ref{PNssprop} carries over, replacing Theorem \ref{Hsstopthm} with Theorem \ref{Hsttopthm}.
\end{proof}

\begin{definition}\label{stRdef}
 Given a    $k$-normal complex $C^*$-algebra $B$, define $ \oR^{\Dol,\st}_{X,x}(B)\subset \oR^{\Dol}_{X,x}(B)$  to  be the subspace consisting  of those $p$ for which $\psi(p)\in \FD\oR^{\Dol,\st}_{X,x}$ for all $\psi \co B \to \Mat_k(\Cx)$.  
\end{definition}

\begin{corollary}\label{PNrepst}
For any $k$-normal complex $C^*$-algebra $B$,   $ \oR^{\Dol,\st}_{X,x}(B)$ is isomorphic to the set of continuous algebra homomorphisms $E^J_{X,x} \to B$.
\end{corollary}
\begin{proof}
The proof of Corollary \ref{PNrepss} carries over, replacing Proposition \ref{PNssprop} with  Proposition \ref{PNstprop}.
 \end{proof}

\subsection{Circle actions and $C^*$-dynamical systems}

\begin{definition}
Define a circle action on a (real or complex)  pro-$C^*$-algebra $A$ to be a continuous group homomorphism from $S^1$ to $\Aut_{\pro(C^*\Alg)}(A)$. Here, the topology on $\Aut_{\pro(C^*\Alg)}(A)$ is defined pointwise, so a net $f_i$ converges to $f$ if and only if $f_i(a) \to f(a)$ for all $a \in A$.   
\end{definition}

The following is immediate:
\begin{lemma}\label{circle1}
Giving a circle action on a $C^*$-algebra $A$ is equivalent to giving a pro-$C^*$-algebra homomorphism $f\co A \to C(S^1,A)$  
satisfying
\begin{enumerate}
 \item $1^*\circ f = \id_A\co A \to C(\{1\},A)=A$;
\item the diagram
\[
 \begin{CD}
  A @>f>> C(S^1,A)\\
@VfVV @VV{C(S^1,f)}V\\
C(S^1,A) @>{m^*}>> C(S^1 \by S^1, A)
 \end{CD}
\]
commutes, where $m \co S^1 \by S^1 \to S^1$ is the multiplication.
\end{enumerate}
\end{lemma}

\begin{lemma}\label{functcircle}
 If a functor $F\co C^*\Alg_k \to \Set$ is represented by a pro-$C^*$-algebra $A$, then to give a circle action on $A$ is equivalent to giving maps
\[
{\alpha}_B\co F(B) \to F(C(S^1,B)),
\]
functorial in $B$, such that
\begin{enumerate}
 \item $F(1^*)\circ {\alpha}_B = \id_{F(B)}\co F(B) \to F(B)$;
\item the diagram
\[
 \begin{CD}
  F(B) @>{\alpha}_B>> F(C(S^1,B))\\
@V{\alpha}_BVV @VV{{\alpha}_{C(S^1,B)}}V\\
F(C(S^1,B)) @>{F(m^*)}>> F(C(S^1 \by S^1, B))
 \end{CD}
\]
commutes, where $m \co S^1 \by S^1 \to S^1$ is the multiplication.
\end{enumerate}
\end{lemma}
\begin{proof}
If $A$ has a circle action ${\alpha}$, then a homomorphism $h\co A \to B$ gives rise to $C(S^1, h)\co  C(S^1,A) \to C(S^1, B)$, and we define ${\alpha}_B(h):= C(S^1,h) \circ {\alpha}$. This clearly satisfies the required properties.

Conversely, given maps ${\alpha}_B$ as above, write $A= \Lim_i A_i$ as an inverse limit of $C^*$-algebras, and let $h_i\co A \to A_i$ be the structure map. Then ${\alpha}_{A_i}(h_i) \in F(C(S^1, A_i))$ is a map $A \to C(S^1, A_i)$. Since the ${\alpha}_{A_i}(h_i)$ are compatible, we may take the inverse limit, giving a  map
\[
 {\alpha}\co A \to C(S^1, A)
\]
To see that this is a group homomorphism, just observe that the conditions above ensure that $h_i \circ 1^*\circ {\alpha} = h_i$ and 
\[
 C(S^1\by S^1, h_i)\circ C(S^1,{\alpha}) \circ {\alpha}= C(S^1\by S^1, h_i) \circ m^*{\alpha} 
\]
for all $i$. Taking the inverse limit over $i$ shows that this satisfies the conditions of Lemma \ref{circle1}.

Finally, note that these two constructions are clearly inverse to each other.
 \end{proof}

\begin{proposition}\label{circleX}
For every  compact K\"ahler manifold $X$, there is a canonical continuous circle action on $E^J_{X,x}$. 
\end{proposition}
\begin{proof}
Given $(U(\sP), D, f) \in \oR^J_{X,x}(B)$, define ${\alpha}(U(\sP),D, f) \in \oR^J_{X,x}(C(S^1, B))$ as follows.

Decompose $D= d^++\vartheta$ into anti-self-adjoint  and self-adjoint parts. Set ${\alpha}(U(\sP)):= C(S^1, U(\sP))= U(\sP)\by_{\sA^0_X(U(B))}\sA^0_X(C(S^1, U(B)))$, and then define ${\alpha}(D):= d^++ t \dmd \vartheta$, where $t \in C(S^1, \Cx)$ is the canonical embedding and $\dmd$ is from Definition \ref{dmd}.  Thus we have constructed ${\alpha}(U(\sP), D, f):= (C(S^1, U(\sP), d^+ + t\dmd \vartheta, C(S^1, f))$, and it is easy to check that this satisfies the conditions of Lemma \ref{functcircle}.
\end{proof}

\begin{remark}\label{pureMHSrk}
By considering finite dimensional quotients of $E^J_{X,x}$, the circle action induces a continuous map
\[                                                                                                          
S^1 \by E^J_{X,x} \to O(\pi_1(X,x)^{\red}_{\R})',
\]
 for $O(\pi_1(X,x)^{\red}_{\R})^{\vee}$ as in
Remark \ref{cfredpi}. This descends to a discontinuous action of $S^1$ on $ O(\pi_1(X,x)^{\red}_{\R})^{\vee}$, as in \cite{Simpson} (made explicit in the real case as \cite[Lemma \ref{mhs2-discreteact}]{mhs2}).

Note, however that the circle action descends to  continuous actions on  $(E^J_{X,x})_{\RFD}, (E^J_{X,x})_{\PN}$ (which  are subalgebras of $ O(\pi_1(X,x)^{\red}_{\R})^{\vee}$, though not closed).

Continuity of the circle action ensures that the map
\[
 S^1 \by \pi_1(X,x) \to E^J_{X,x}
\]
is continuous, and hence that the induced map $ S^1 \by \pi_1(X,x) \to \pi_1(X,x)^{\red}_{\R}(\R) \subset O(\pi_1(X,x)^{\red}_{\R})^{\vee}$ is continuous. Thus a continuous circle action on  $E^J_{X,x}$ gives rise to a pure Hodge structure on $ \pi_1(X,x)^{\red}$ in the sense of \cite[\S 5]{Simpson}, but without needing to refer to $\pi_1(X,x)$ itself. This suggests that the most natural definition of a pure non-abelian Hodge structure is a continuous circle action on a pro-$C^*$-bialgebra.
\end{remark}

\begin{example}\label{circleEJabgp}
Lemma \ref{ablemma} gives an isomorphism
\[
(E^J_{X,x})^{\ab} = \{f \in C(\H^1(X, \Cx^*),\Cx)\,:\, f(\bar{\rho})= \overline{f(\rho)}\},
\]
and 
\ref{EJabgp} then shows that the grouplike elements are $G((E^J_{X,x})^{\ab}) \cong \H_1(X,\Z\oplus\R)$. To describe the circle action on $(E^J_{X,x})^{\ab} $, it thus suffices to describe it on the space $\H^1(X, \Cx^*) $ of one-dimensional complex representations. 

Taking the decomposition $D= d^++\vartheta$ of  a flat connection $D$ into anti-hermitian and hermitian parts, note that we must have $(d^+)^2= \vartheta^2=0$, because commutativity of $\Cx^*$ ensures that commutators vanish, everything else vanishing by hypothesis. This decomposition therefore corresponds to the isomorphism $\H^1(X, \Cx^*)\cong \H^1(X,\S^1) \by \H^1(X,\R)$. Since the action is given by $\vartheta \mapsto t \dmd \vartheta$ for $t \in S^1$, it follows that the $S^1$-action is just the $\dmd$-action on $\H^1(X,\R)$.

On $G((E^J_{X,x})^{\ab}) \cong \H_1(X,\Z)\oplus \H_1(X,\R)$, this means that the circle action fixes $\H_1(X,\Z)$ and acts with the $\dmd$-action on $\H_1(X,\R)= \H_1(X,\R)^{\vee}$.
\end{example}

\begin{definition}\label{DSdef}
 Recall from \cite[Definition 2.6]{williamsXprod} that a  $C^*$-dynamical system  is a triple $(A, G, \alpha)$, for $G$ a locally compact topological group,  $A$ a $C^*$-algebra, and  $\alpha$ a continuous action of $G$ on a  $A$.
\end{definition}

\begin{lemma}
 The circle action ${\alpha}$ of Proposition \ref{circleX} gives rise to  a pro-$C^*$-dynamical system $(E^J_{X,x}, S^1, {\alpha})$, i.e. an inverse system of $C^*$-dynamical systems.
\end{lemma}
\begin{proof}
 Since $E^J_{X,x}$ is a pro-$C^*$-algebra, we may write it as an inverse system $E^J_{X,x}= \Lim_i E_i$, for $C^*$-algebras $E_i$. The circle action then sends the structure map $h_i \co E^J_{X,x} \to E_i$ to the map $C(S^1, h_i) \circ {\alpha} \co E^J_{X,x} \to C(S^1,E_i)$, and evaluation at $1 \in S^1$ recovers $h_i$. We may therefore set $E_{{\alpha}(i)}$ to be the closure of the image of $E^J_{X,x} \to C(S^1,E_i)$, and observe that $E_{{\alpha}(i)}$ is $S^1$-equivariant, with $E^J_{X,x}= \Lim_i E_{{\alpha}(i)}$. 

Thus $(E^J_{X,x}, S^1, {\alpha})= \Lim_i (E_{{\alpha}(i)},S^1, {\alpha})$ is a pro-$C^*$-dynamical system.
\end{proof}

The following is taken from \cite[Lemma 2.27]{williamsXprod}:
\begin{definition}
 Given a  $C^*$-dynamical system  $(A, G, \alpha)$ and $f \in C_c(G,A)$, define 
\[
 \|f\|:= \sup\{ \|\pi \rtimes U(f)\|\co (\pi, U) \text{ a covariant representation of } (A, G, \alpha)\}.
\]
Then $\|-\|$  is called the universal norm, and dominated by $\|-\|_1$.

The completion of $C_c(G,A)$ with respect to $\|-\|$ is the crossed product of $A$ by $G$, denoted $A \rtimes_{\alpha}G$.
\end{definition}

\begin{definition}
Define a polarised real Hilbert variation of Hodge structures  of weight $n$ on $X$ to be a real  local system $\vv$, with a pluriharmonic metric on $\sA^0_X(\vv)$,  equipped with a Hilbert space decomposition
\[
 \sA^0_X(\vv)\ten \Cx = \hat{\bigoplus}_{p+q=n} \sV^{pq},       
\]
(where $\hat{\bigoplus}$ denotes Hilbert space direct sum), with $\overline{\sV^{pq}}= \sV^{qp}$, and satisfying the conditions
\[
\pd: \sV^{pq} \to \sV^{pq}\ten_{\sA^0_X(\Cx)}\sA^{10}_X, \quad \bar{\theta}: \sV^{pq} \to \sV^{p+1,q-1}\ten_{\sA^0_X(\Cx)}\sA^{01_X},
\]
for the decomposition
 $D= \pd +\bar{\pd}+\theta + \bar{\theta}$ of Definition \ref{Ddecomp}.
\end{definition}

\begin{proposition}\label{VHSsemidirect}
Real Hilbert space representations of the non-unital pro-$C^*$-algebra  $E^J_{X,x}\rtimes_{\alpha}S^1$ correspond to framed weight $0$ polarised real Hilbert variations of Hodge structure.
\end{proposition}
\begin{proof}
By \cite[Proposition 2.29]{williamsXprod}, a  $*$-representation $E^J_{X,x}\rtimes_{\alpha}S^1 \to L(H) $ for a Hilbert space $H$ consists of:
\begin{enumerate}
 \item a  $*$-representation $\rho \co E^J_{X,x} \to L(H)$, and
\item a continuous representation $u\co S^1 \to U(H)$
\end{enumerate}
such that
\[
 \rho({\alpha}(t,a)) = u(t)\rho(a)u(t)^{-1}       
\]
for all $a \in E^J_{X,x}, t \in S^1$.

In other words, in $\oR^J_{X,x}(C(S^1, L(H)))$, we have $\alpha(\rho)= u\rho u^{-1}$, so $\alpha(\rho)$ and $\rho$ are isomorphic in the groupoid $\cR^J_X(C(S^1, L(H)))$.

Now, by definition of $E^J_{X,x} $, the representation $\rho$ corresponds to a real  local system $\vv$,  with a pluriharmonic metric on $\sA^0_X(\vv)$ and a Hilbert space isomorphism $f\co \vv_x \to H$. The representation $\alpha(\rho)$ corresponds to the connection   $\alpha(D):= d^++ t \dmd \vartheta$ on $\sA^0_X(C(S^1,\vv))$ for the standard co-ordinate $t\co S^1 \to \Cx$,  together  with framing $f$. 

The condition that $\alpha(\rho)$ and $\rho$ are isomorphic then gives us a unitary gauge transformation $g$ between them. In other words, we have a continuous representation $g \co S^1 \to \Gamma(X, U(\sA^0_X(\vv)))$ with $\alpha(D) \circ g= g \circ D$. We must also have $g_x=u$. 

Thus $g$ gives us a Hilbert space decomposition
\[
 \sA^0_X(\vv)\ten \Cx = \hat{\bigoplus}_{p+q=0} \sV^{pq},       
\]
  with $\overline{\sV^{pq}}= \sV^{qp}$, and $g(t)$  acting on $\sV^{pq}$ as multiplication by $t^{p-q}$. The condition  $\alpha(D) \circ g= g \circ D$ then forces the   
conditions
\[
\pd: \sV^{pq} \to \sV^{pq}\ten_{\sA^0_X(\Cx)}\sA^{10}_X, \quad \bar{\theta}: \sV^{pq} \to \sV^{p+1,q-1}\ten_{\sA^0_X(\Cx)}\sA^{01_X},
\]
as required.     
\end{proof}

\begin{remark}
Given any $E^J_{X,x}$-representation $V$, \cite[Example 2.14]{williamsXprod} gives an $E^J_{X,x} \rtimes S^1$-representation  $\Ind_e^{S^1}V$. Its  underlying Hilbert space is just the space $L^2(S^1, V)$ of $L^2$-measurable $V$-valued forms on the circle with respect to Haar measure. For the pluriharmonic local system $\vv$ associated to $V$, this therefore gives a weight $0$ variation $\Ind_e^{S^1}\vv$ of Hodge structures on $X$, with 
$\sA^0_X(\Ind_e^{S^1}\vv)= \sA^0_X(L^2(S^1, \vv))$.
\end{remark}

\section{Hodge decompositions on cohomology}\label{cohosn}

Fix a compact K\"ahler manifold $X$.

\begin{definition}
Given a pluriharmonic local system  $\vv$ in real Hilbert spaces on $X$ (as in Example \ref{plurilocsys}), the inner product on $\vv$ combines with the
K\"ahler metric on $X$ to give  inner products $\<-,-\>$
on the spaces  $A^n(X,\vv)$ for all $n$. Given an operator $F$ on $A^*(X,\vv)$, we denote the adjoint operator by $F^*$. Let $\Delta= DD^*+D^*D$.
\end{definition}

\subsection{Sobolev spaces}

Note that in general, the Laplacian $\Delta$ is not a bounded operator in the $L^2$ norm. We therefore introduce a system of Sobolev norms:

\begin{definition}
Define  $L^n_{(2),s}(X, \vv)$ to be the completion of $A^n(X,\vv)$ with respect to the inner product $\<v,w\>_s:= \<v, (I + \Delta)^sw\>$.
\end{definition}

Note that we then have bounded operators $D,D^c \co L^n_{(2),s}(X, \vv)\to L^{n+1}_{(2),s-1}(X, \vv)$, $D^*,D^{c*}\co  L^n_{(2),s}(X, \vv)\to L^{n-1}_{(2),s-1}(X, \vv)$ and $\Delta\co L^n_{(2),s}(X, \vv)\to L^n_{(2),s-2}(X, \vv)$.

\begin{proposition}\label{sobolevtower}
The maps  $(I+\Delta)^k \co L^n_{(2),s}(X, \vv)\to L^n_{(2),s-2k}(X, \vv)$ are Hilbert space  isomorphisms, and there are canonical  inclusions $L^n_{(2),s}(X, \vv)\subset L^n_{(2),s-1}(X, \vv)$.
\end{proposition}
\begin{proof}
 The proofs of \cite[Proposition 2.3 and Lemma 2.4]{dodziuk} carry over to this generality.
\end{proof}

\begin{definition}
Define $\cH^n (X,\vv) \subset A^n(X,\vv)$ to consist of forms with $\Delta\alpha=0$. Regard this as a pre-Hilbert space with the inner product $\<-,-\>$.
\end{definition}
The following implies that $\cH^n (X,\vv)$ is in fact a Hilbert space:

\begin{lemma}
The inclusions
\[
 \cH^n (X,\vv) \to \{\alpha \in L^n_{(2),0}\,:\, D\alpha =D^*\alpha=0\} \to \{\alpha \in L^n_{(2),0}\,:\, \Delta\alpha=0\}
\]
are Hilbert space isomorphisms.
 \end{lemma}
\begin{proof}
 When  $\vv$ is the local system associated to the $\pi_1(X)$-representation $\ell^2(\pi_1(X,x))$, this is \cite[Lemma 2.5]{dodziuk}, but the same proof carries over. 
\end{proof}

\subsubsection{Decomposition into eigenspaces}

\begin{definition}
 Define  $T$ to be the composition of $L^n_{(2),0}(X, \vv)\xra{(I+\Delta)^{-1} }  L^n_{(2),2}(X, \vv)\into L^n_{(2),0}(X, \vv)$. This is bounded and self-adjoint, with spectrum $\sigma(T) \subset (0,1]$. Thus  the spectral decomposition gives $T= \int_{(0,1]} \lambda d\pi_{\lambda}$ for some projection-valued measure $\pi$ on $(0,1]$.
\end{definition}
 
For $S\subset [0, \infty)$ measurable, write 
\[
 \nu(S):= \pi_{\{(1+\rho)^{-1}\,:\, \rho \in S\}}.
\]
Thus
\[
 T= \int_{\rho\in [0, \infty)} (1+\rho)^{-1} d\nu_{\rho},
\]
and for  $v \in \cL_{(2),s+2}(X,\vv)$ we have
\[
 \Delta v = \int_{\rho\in [0, \infty)} \rho d\nu_{\rho}v \in \cL_{(2),s}(X,\vv).
\]

If we set $E^n(S):= \nu(S)\cL_{(2),0}^n(X,\vv)$, observe that $E^n$ defines a measurable family of Hilbert spaces on $[0, \infty)$, 
 and that we have direct integral decompositions 
\[
  \cL_{(2),s}^n(X,\vv)= \int^{\oplus}_{\rho \in [0, \infty)} (1+\rho)^{-s/2} E^n_{\rho}.
\]

\subsubsection{Harmonic decomposition of eigenspaces}\label{eigensn}

Since the operators $D,D^c, D^*, D^{c*}$ commute with $\Delta$, they descend to each graded Hilbert space $E^n(S)$, provided $S$ is bounded above. 

If $S$  also has a strictly positive lower bound, then $\Delta$ is invertible on $E^n(S)$, so
\[
 E^n(S)= \Delta E^n(S).
\]
 As $\Delta = DD^*+D^*D=D^cD^{c*}+D^{c*}D^c$, with $[D,D^c]=[D^*,D^c]=[D^{c*},D]=0$, this implies that 
\begin{align*}
 E^n(S)&= DE^{n-1}(S)\oplus D^*E^{n+1}(S)= D^cE^{n-1}(S)\oplus D^{c*}E^{n+1}(S)\\
&= DD^cE^{n-2}(S)\oplus D^*D^cE^n(S)\oplus DD^{c*}E^n(S)\oplus D^*D^{c*}E^{n+2}(S).
\end{align*}
Furthermore, $D\co D^*E^n(S)\to DE^{n-1}(S)$ and $D^*\co DE^n(S)\to D^*E^{n+1}(S)$ are isomorphisms, with similar statements for $D^c$.

If $S$ is just bounded above and does not contain $0$, then the statements above still hold if we replace subspaces with their closures:

\begin{proposition}\label{Hodgedecomp0}
 There are Hilbert space decompositions
\begin{eqnarray*}
\cL^n_{(2),s}(X, \vv) &= &\cH^n(X,\vv)\oplus\overline{\Delta \cL^n_{(2),s+2}(X, \vv) } \\
 &= &\cH^n(X,\vv)\oplus\overline{D\cL^{n-1}_{(2),s+1}(X, \vv)}\oplus \overline{D^*\cL^{n+1}_{(2),s+1}(X, \vv)}\\
&=&\cH^n(X,\vv)\oplus \overline{D^c\cL^{n-1}_{(2),s+1}(X, \vv)}\oplus \overline{D^{c*}\cL^{n+1}_{(2),s+1}(X, \vv)}\\
&=&\cH^n(X,\vv)\oplus \overline{DD^c\cL^{n-2}_{(2),s+2}(X, \vv)} \oplus\overline{D^*D^c\cL^{n}_{(2),s+2}(X, \vv)}\\
&&\phantom{\cH^n(X,\vv)} \oplus \overline{DD^{c*}\cL^{n}_{(2),s+2}(X, \vv)} \oplus\overline{D^*D^{c*}\cL^{n+2}_{(2),s+2}(X, \vv)}.
\end{eqnarray*}
for all $s$.
\end{proposition}
\begin{proof}
This is essentially the Hodge Theorem, and we can construct the decomposition by a slight modification of  \cite[pp94--96]{GriffithsHarris}. 

We can define an  approximate Green's function by
\[
 G_{\eps}:= \int_{(0,1-\eps]} \frac{\lambda}{1-\lambda} d\pi_{\lambda}.
\]
Now, note that $(I+\Delta)G_{\eps}= \int_{(0,1-\eps]} \frac{1}{1-\lambda} d\pi_{\lambda}$, which is bounded, so $G_{\eps}$ is the composition of the inclusion $ L^n_{(2),s+2}(X, \vv)\into L^n_{(2),s}(X, \vv)$ with a map
\[
 G_{\eps}\co L^n_{(2),s}(X, \vv)\to  L^n_{(2),s+2}(X, \vv).
\]

Also note that $G_{\eps}$ commutes with $\Delta$, and that $\Delta G_{\eps} = \pi(0,1-\eps]= I- \pi[1-\eps,1]$. As $\eps \to 0$, this means that $\pi(1) +\Delta G_{\eps}$ converges weakly to $I$. Since $\pi(1)$ is projection onto $\cH^n(X,\vv)$, this gives the decomposition required (noting that norm closure and weak closure of a subspace are the same, by the Hahn--Banach Theorem).
\end{proof}

Now, for $v \in   D^*E^n(S)$, $\<Dv, Dv\>= \<\Delta v, v\>$, so $\|Dv\|^2/\|v\|^2$ lies in $S$. 

\begin{definition}\label{Deltahalf}
 Define $\Delta^{\half}\co \cL_{(2),s+1}^n \to  \cL_{(2),s}^n$ by
\[
 \Delta^{\half}:= \int_{\rho\in [0, \infty)} \rho^{\half} d\nu_{\rho}.
\]
\end{definition}

This gives
\begin{align*}
D\cL_{(2),s+1}^{n-1}(X,\vv)=  \Delta^{\half}\overline{D\cL_{(2),s+2}^{n-1}(X,\vv)}, \\
\ker D \cap  \cL_{(2),s}^n(X,\vv)= \overline{D\cL_{(2),s+1}^{n-1}(X,\vv)}  \oplus \cH^n(X, \vv).
\end{align*}
Thus 
\[
 \H^n\cL_{(2),s}^{\bt}(X,\vv)\cong \cH^n(X, \vv) \oplus (\overline{D\cL_{(2),s-1}^{n-1}(X,\vv)}/\Delta^{\half}\overline{D\cL_{(2),s-2}^{n-1}(X,\vv)}).
\]
There are similar statements for the operators $D^c,D^*, D^{c*}$.

\begin{definition}\label{Deltahalf2}
 Define $\Delta^{-\half}D\co \overline{D^*\cL_{(2),s+1}^n(X,\vv)} \to  \cL_{(2),s(X,\vv)}^n$ by
\[
 \Delta^{-\half}:= \int_{\rho\in (0, \infty)} \rho^{-\half} d\nu_{\rho} \circ D,
\]
and define $\Delta^{-\half}D^c\co \overline{D^{c*}\cL_{(2),s+1}^n(X,\vv)} \to  \cL_{(2),s}^n(X,\vv)$ similarly.
\end{definition}

\begin{proposition}\label{DeltahalfProp}
 The operator $\Delta^{-\half}D$ gives a Hilbert space isomorphism from the closed subspace $\overline{D^*\cL_{(2),s+1}^n(X,\vv)} $ of $\cL_{(2),s}^{n-1}(X,\vv)$ to the closed subspace $\overline{D\cL_{(2),s+1}^{n-1}(X,\vv)}  $ of $\cL_{(2),s}^{n}(X,\vv) $. 

Likewise, $\Delta^{-\half}D^c\co \overline{D^{c*}\cL_{(2),s+1}^n(X,\vv)} \to \overline{D^c\cL_{(2),s+1}^{n-1}(X,\vv)}$ is a Hilbert space isomorphism.
\end{proposition}
\begin{proof}
We prove this for the first case, the second being entirely similar.
Given  $a, b \in D^*A^n(X,\vv)$, we have
\begin{align*}
 \<\Delta^{-\half}Da, \Delta^{-\half}Db\>_s &= \< (I+\Delta)^s \Delta^{-\half}Da, \Delta^{-\half}Db\>\\
&= \< (I+\Delta)^s \Delta^{-1} D^*Da, b\>\\
&= \< (I+\Delta)^s a, b\>= \<a,b\>_s,
\end{align*}
since $D^*a=0$ gives $D^*Da = \Delta a$.
Taking Hilbert space completions with respect to  $\<-,-\>_s$ then gives the required result.
\end{proof}

\subsection{The Hodge decomposition and  cohomology}

\begin{proposition}\label{sobolevlemma}
The nested intersection $\bigcap_s L^p_{(2),s}(X,\vv) $ is the space $A^p(X,\vv)$ of $\C^{\infty}$ $\vv$-valued $p$-forms.  
\end{proposition}
\begin{proof}
When $\vv$ is finite-dimensional, this is the Global Sobolev Lemma, but the same proof applies for Hilbert space coefficients.  
\end{proof}

\begin{theorem}\label{Hodgedecomp}
 There are pre-Hilbert space decompositions
\begin{align*}
 A^n(X, \vv)&=  \cH^n(X,\vv)&\oplus&\overline{\Delta A^n(X, \vv) }\\
 &=  \cH^n(X,\vv)&\oplus&\overline{DA^{n-1}(X, \vv)}\oplus \overline{D^*A^{n+1}(X, \vv)}\\
&=  \cH^n(X,\vv)&\oplus&\overline{D^cA^{n-1}(X, \vv)}\oplus \overline{D^{c*}A^{n+1}(X, \vv)}\\
&=  \cH^n(X,\vv)&\oplus&\overline{DD^cA^{n-2}(X, \vv)} \oplus\overline{D^*D^cA^{n}(X, \vv)}\oplus \overline{DD^{c*}A^{n}(X, \vv)} \oplus\overline{D^*D^{c*}A^{n+2}(X, \vv)}.
\end{align*}
for all $n$.
 \end{theorem}
\begin{proof}
We just take the inverse limit $\Lim_s$  of the decomposition  in Proposition \ref{Hodgedecomp0}, and then make the substitution of Proposition \ref{sobolevlemma}. 
\end{proof}

\subsubsection{Reduced cohomology}\label{redcohosn}

\begin{definition}
Given a cochain complex $C^{\bt}$ in topological vector spaces, write
\[
\bar{\H}^n(C^{\bt}):= \H^n(C^{\bt})/\overline{\{0\}}, 
\]
where $\overline{\{0\}}$ is the closure of $0$. Note that we could equivalently define $\bar{\H}^*$ as the quotient of the space of cocycles by the \emph{closure} of the space of coboundaries.

Given a  local system  $\vv$ in topological vector spaces on $X$, define
\[
 \bar{\H}^n(X, \vv):= \bar{\H}^n(A^{\bt}(X, \vv)).
\]
\end{definition}

\begin{corollary}\label{harmcoho}
 The 
maps
\[
 \cH^n(X,\vv) \to \bar{\H}^n(X,\vv)
\]
 are all topological isomorphisms.
\end{corollary}

\begin{corollary}[Principle of two types]\label{p2t}
As subspaces of  $A^n(X, \vv)$,
\[
\ker D \cap \ker D^c \cap (\overline{ DA^{n-1}(X,\vv) } + \overline{D^cA^{n-1}(X,\vv) }) =  \overline{DD^cA^{n-2}(X,\vv)}.
\]
\end{corollary}

\begin{lemma}[Formality]\label{formallemma}
The morphisms
\[
 (\bar{\H}_{D^c}^*(X,\vv),0) \la (\z_{D^c}^*(X,\vv), D) \to (A^*(X,\vv), D)
\]
induce isomorphisms on reduced cohomology.
\end{lemma}
\begin{proof}
 The proof of \cite[Lemma 2.2]{Simpson} carries over to this generality.
\end{proof}

\begin{remark}
Usually, formality statements such as Lemma \ref{formallemma} lead to isomorphisms on deformation functors (see \cite{GM} for the original case, and \cite[Proposition \ref{htpy-formalpins}]{htpy} 
for the case closest to our setting). 

However, there does not appear to be a natural deformation functor associated to topological DGLAs $L$ with obstruction space $\bar{\H}^2(L)$. Thus, in contrast to the  pro-algebraic case, it is not clear whether there are natural completions of the homotopy groups which can be described in terms of the reduced cohomology ring.

The description of the Archimedean monodromy in \cite[Theorem \ref{mhs2-archmonthm}]{mhs2}
is even less likely to adapt, since  it features the Green's operator $G$, which we have had to replace with a non-convergent sequence of operators.
 \end{remark}

\subsubsection{Non-reduced cohomology}\label{nonredcohosn}

Taking the inverse limit $\LLim_s$ of the decompositions of \S \ref{eigensn}, we obtain
\begin{align*}
 DA^{n-1}(X,\vv) &= \Delta^{\half} \overline{DA^{n-1}(X,\vv)},\\
D^*A^{n+1}(X,\vv) &= \Delta^{\half} \overline{D^*A^{n+1}(X,\vv)},\\
\end{align*}
with similar statements for $D^c, D^{c*}$.

Thus:
\begin{proposition}\label{nonredcohoprop}
\[
\H^nA^{\bt}(X,\vv) \cong  \cH^n(X, \vv) \oplus(\overline{DA^{n-1}(X,\vv)} /\Delta^{\half}\overline{DA^{n-1}(X,\vv)} ).
\]      
\end{proposition}

Applying the operator $*$ then gives
\[
  \H^nA^{\bt}(X,\vv)\oplus_{\cH^n(X, \vv)} \H^{2d-n}A^{\bt}(X,\vv')\cong A^n(X,\vv)/\Delta^{\half}A^n(X,\vv).
\]

Moreover, we have topological isomorphisms
\begin{align*}
\Delta^{-\half}D\co \overline{D^*A^{n}(X,\vv)} &\to  \overline{DA^{n-1}(X,\vv)},\\
\Delta^{-\half}D^c\co \overline{D^{c*}A^{n}(X,\vv)} &\to  \overline{D^cA^{n-1}(X,\vv)},
\end{align*}
for $\Delta^{-\half}D,\Delta^{-\half}D^c$ as in Definition \ref{Deltahalf2}. 

\subsection{The $W^*$-enveloping algebra}

\subsubsection{ $E^J(X,x)'$ }\label{Evee}

\begin{definition}
Given a  $C^*$-algebra $B$ and a positive linear functional $f$, define $B_f$ to be the Hilbert space completion of $B$ with respect to the bilinear form $\<a,b\>_f:= f(a^*b)$. We define $\pi_f$ to be the representation of $B$ on $B_f$ by left multiplication. Note that this is a cyclic representation, generated by $1 \in B_f$. 
\end{definition}

\begin{lemma}\label{Edual}
Given a $C^*$-algebra $B$, the topological dual is given by $B'= \LLim_f B_f'$, where $f$ ranges over the filtered  inverse system of positive linear functionals on $B$.
\end{lemma}
\begin{proof}
 This amounts to showing that $B^{\vee\vee} = \Lim_f B_f$. Since the system is filtered (with $f+g \ge f,g$  and $B_g \to B_f$ for $g \ge f$), $\hat{B}:=\Lim_f B_f$ is the completion of $B$ with respect to the seminorms $\|b\|_f:= f(b^*b)^{\half}$. The space $B_f$ is the strong closure of 
 $B$ in the cyclic representation $\pi_f$, which is just the image of $B^{\vee\vee}$, by the von Neumann bicommutant theorem. Since $B^{\vee\vee}$ is the completion of $B$ with respect to the system of weak seminorms for all representations, this implies that   the map $\hat{B} \to B^{\vee\vee}$ is an equivalence. 
\end{proof}

\begin{lemma}
For a $C^*$-algebra $B$, and a   $B$-representation $V$ in Hilbert spaces, 
\[
 \Hom_B(V, B') \cong V'.
\]
\end{lemma}
\begin{proof}
 The space $\Hom_B(V, B') $ consists of continuous $B$-linear maps $V \to B'$, and hence to continuous $(A,k)$-bilinear  maps $A\by V \to k$. These correspond to  continuous linear maps $V \to k$, as required
\end{proof}

Considering smooth morphisms from $X$ then gives:
\begin{corollary}\label{univlocsys}
For any $E:=E^J(X,x)$-representation $V$ in real Hilbert spaces, with corresponding  local system $\vv$ on $X$, there is a canonical topological isomorphism
\[
 A^{\bt}(X, \vv) \cong \Hom_E(V', A^{\bt}(X, \bE')),
\]
where $\bE'$ is the direct system of pluriharmonic local systems corresponding to the ind-$E$-representation $E'$ given by left multiplication.
 \end{corollary}

Of course, all the cohomological decomposition results of this section extend to direct limits, so apply to $\bE'$.
Conversely, Corollary \ref{univlocsys}  all such results for local systems $\vv$ can be inferred from the corresponding results with  $\bE'$-coefficients.

\begin{remark}\label{DGArk}
The comultiplication $E^J(X,x)\to  E^J(X,x)\hten E^J(X,x)$ of Lemma \ref{tensorstr} induces a multiplication
\[
E^J(X,x)'\bar{\ten} E^J(X,x)' \to E^J(X,x)'       
\]
on continuous duals, where $ (\Lim_i E_i)'\bar{\ten}(\Lim_jE_j)':= \LLim_{i,j} E_i \bar{\ten} E_j $ for $C^*$-algebras $E_i$ and the  dual tensor product $\bar{\ten}$ of \cite[p. 210]{takesakiTOA}. In particular, $\bar{\ten}$ is a crossnorm, so we have a jointly continuous multiplication on $E^J(X,x)'$

Thus  $A^{\bt}(X, \bE')$ is also equipped with a jointly continuous (graded) multiplication, so has the structure of a differential graded topological algebra.       
\end{remark}

\subsubsection{Failure of continuity}

Since direct integrals of harmonic representations must be harmonic,
Corollary \ref{harmcoho} and Proposition \ref{nonredcohoprop} provide us with information about the behaviour of cohomology in measurable families. In particular, they allow us to recover space of measures on the topological  spaces of cohomology groups fibred over the moduli spaces of local systems. Thus $E^J_{X,x}$ is  a much finer invariant than the pro-algebraic completion of $\pi_1(X,x)$.

It is natural to ask whether we can strengthen the Hodge decomposition to incorporate finer topological data. The following example indicates that it does not hold for coefficients in $\bE$ itself:

\begin{example}
Let $X$ be a complex torus, so $\pi_1(X,e) \cong \Z^{2g}$. By Proposition \ref{abcor}, $(E^J_{X,e})_{\RFD}\ten \Cx = C(\Hom(\pi_1(X,e), \Cx^*),\Cx) \cong C( (\Cx^*)^{2g}, \Cx)$. 

The complex $A^{\bt}(X, \bE_{\RFD}\ten \Cx)$ is then quasi-isomorphic to $\H^*(\Z^{2g}, C( (\Cx^*)^{2g}, \Cx))$. 
This is given by taking the completed tensor product of $2g$ copies of the complex $F$ given by $C(\Cx^*,\Cx) \xra{z-1} C(\Cx^*,\Cx)$, so
\[
 \bar{\H}^*  (X, \bE_{\RFD}\ten \Cx)\cong \Cx[-2g].     
\]

However, $D^*D+DD^*= |z-1|^2$ on the complex $F$, so harmonic forms are given by
\[
 \cH^*  (X, \bE_{\RFD}\ten \Cx)\cong 0.      
\]
\end{example}

\section{Twistor and Hodge structures on cochains, and $\SU_2$}\label{twistorHodgecohosn}

\subsection{Preliminaries on non-abelian twistor and Hodge filtrations}

The following is \cite[Definition \ref{mhs2-Cdef}]{mhs2}: 
\begin{definition}\label{cdef}
Define $C$ to be the real affine scheme $\prod_{\Cx/\R}\bA^1$ obtained from $\bA^1_{\Cx}$ by restriction of scalars, so for any real algebra $A$, $C(A)= \bA^1_{\Cx}(A\ten_{\R}\Cx)\cong A\ten_{\R}\Cx$. Choosing $i \in \Cx$ gives an isomorphism $C \cong \bA^2_{\R}$, and we let $C^*$ be the quasi-affine scheme $C - \{0\}$.

We let the real algebraic group $S=\prod_{\Cx/\R} \bG_m$ of Definition \ref{Sdef} act on $C$ and $C^*$
by inverse multiplication, i.e.
\begin{eqnarray*}
S \by C &\to& C\\
(\lambda, w) &\mapsto& (\lambda^{-1}w).
\end{eqnarray*}
 \end{definition}

Fix an isomorphism $C \cong \bA^2$, with co-ordinates $u,v$ on $C$ so that the isomorphism $C(\R) \cong \Cx$ is given by  $(u,v) \mapsto u+iv$. Thus the algebra $O(C)$ associated to $C$ is the polynomial ring $\R[u,v]$. $S$ is isomorphic to the scheme $\bA^2_{\bR} -\{(\alpha,\beta)\,:\, \alpha^2+\beta^2=0\}$, with the group isomorphism $S(\R) \cong \Cx^*$ given by $(\alpha,\beta) \mapsto \alpha+i\beta$, and the group isomorphism $S(\Cx) \cong (\Cx^*)^2$ given by $(\alpha,\beta) \mapsto (\alpha+i\beta, \alpha-i\beta)$.

By \cite[Corollary \ref{mhs2-flathfil} and Proposition \ref{mhs2-flattfil}]{mhs2}, 
real Hodge filtrations  (resp. real twistor structures) correspond to $S$-equivariant (resp. $\bG_m$-equivariant) flat vector bundles on $C^*$. The latter arises because  $[C^*/\bG_m] \simeq \bP^1_{\R}$, so $\bG_m$-equivariant sheaves on $C^*$ correspond to sheaves on $\bP^1$.

The following is \cite[Definition \ref{mhs2-rowdef}]{mhs2}:
\begin{definition}\label{rowdef}
Define an $S$-action on the real affine scheme $\SL_2$   by    
$$
(\alpha,\beta, A) \mapsto\left(
\begin{smallmatrix}  1 & 0  \\ 0 & \alpha^2+\beta^2  \end{smallmatrix} \right)A  \left(
\begin{smallmatrix} \alpha & \beta \\ -\beta & \alpha \end{smallmatrix} \right)^{-1}=
 \left(
\begin{smallmatrix}  \alpha^2+\beta^2 & 0  \\ 0 & 1  \end{smallmatrix} \right)^{-1}A  \left(
\begin{smallmatrix} \alpha  & -\beta \\ \beta & \alpha \end{smallmatrix} \right).
$$ 
Let $\row_1 :\SL_2 \to C^*$ be the $S$-equivariant map given by projection onto the first row. 
\end{definition}

The subgroup scheme $\bG_m \subset S$ is given by $\beta=0$ in the co-ordinates above, and there is a subgroup scheme $S^1 \subset S$ given by $\alpha^2+\beta^2=1$. These induce an isomorphism $(\bG\by S^1)/(-1,-1)\cong S$. On these subgroups, the action on $\SL_2$ simplifies as follows:
\begin{lemma}\label{rowlemma}
The action of $\bG_m \subset S$ on  $\SL_2$ is given   by
\[
 (\alpha, A) \mapsto\left(
\begin{smallmatrix}  \alpha^{-1} & 0  \\ 0 & \alpha \end{smallmatrix} \right)A  
\]
and the action of $S^1 \subset \bG_m$ is given   by
\[
 (\alpha,\beta, A) \mapsto
A  \left(
\begin{smallmatrix} \alpha & \beta \\ -\beta & \alpha \end{smallmatrix} \right)^{-1}.
\]
\end{lemma}

The action of $S^1 \subset S$ descends via the maps above to an action on $\bP^1_{\R}$, which is just given by identifying $S^1$ with the real group scheme $\SO_2$. 

\subsection{The twistor structure on cochains}

Fix a compact K\"ahler manifold $X$.

\begin{definition}
 Define $\tD\co A^n(X, \vv)\ten\O_{C^*}\to A^{n+1}(X, \vv)\ten\O_{C^*}$  by 
\[
 \tD= uD+vD^c,
\]
and write $\tilde{A}^{\bt}(X, \vv) $ for the resulting complex. Put a $\bG_m$-action on $\tilde{A}^{\bt}(X, \vv) $ by letting $A^n(X, \vv) $ have weight $n$,  and giving $\O_{C^*}$ the action of $\bG_m \subset S$ from Definition \ref{cdef}.

Define  $\tilde{\z}^n(X,\vv):= \ker(\tD \co A^n(X,\vv)\ten\O_{C^*} \to A^{n+1}(X,\vv)\ten\O_{C^*})$, $\tilde{\b}^n(X,\vv):= \im(\tD \co A^{n-1}(X,\vv)\ten\O_{C^*} \to A^n(X,\vv))\ten\O_{C^*} $ and
\[
\bar{\tilde{\H}}^n(X,\vv):= \tilde{\z}^n(X,\vv)/\overline{\tilde{\b}^n(X,\vv)}.
\]
\end{definition}

By analogy with \cite[Proposition \ref{mhs2-flattfil} and Theorem \ref{mhs2-mtsmal}]{mhs2}, 
we regard the $\bG_m$-equivariant complex  $\tilde{A}^{\bt}(X, \vv)$ over $C^*$ as a twistor filtration on $A^n(X, \vv)$.

\begin{corollary}\label{harmcoho2}
 The canonical inclusion $\cH^n (X,\vv)(n)\ten\O_{C^*} \to \bar{\tilde{\H}}^n(X,\vv)$ is a $\bG_m$-equivariant topological isomorphism.
\end{corollary}
\begin{proof}
It suffices to prove this on pulling back along the flat cover $\row_1\co \SL_2 \to C^*$. We may define
 $\tD^*\co  A^n(X, \vv)\ten O(\SL_2)\to A^{n-1}(X, \vv)\ten O(\SL_2)$ by $\tD^*= yD^*-xD^{c*}$.
Then $[\tD, \tD^*]=\Delta$, and since $\Delta$ commutes with $D$ and $D^c$ it also commutes with $\tD$. 

The result now follows with the same proof as that of Corollary \ref{harmcoho}, replacing $D,D^*$ with $\tD, \tD^*$.
\end{proof}

\begin{proposition}\label{finecoho}
If we write $\cH^n= \cH^n (X,\vv)$ and $M^m= \overline{DD^cA^{m-2}(X,\vv)}$, then there is a $\bG_m$-equivariant isomorphism
\[
 \tilde{\H}^n(X,\vv) \cong  [(\cH^n \oplus M/\Delta^{\half}M^n)(n) \oplus (M^{n+1}/ \Delta^{\half}M^{n+1})(n-1)] \ten\O_{C^*}      
\]
of quasi-coherent sheaves on $C^*$.
\end{proposition}
\begin{proof}
 Writing $A^m:= A^m(X,\vv)$, we have a  
commutative diagram 
\[
 \xymatrix@R=0ex{ & \overline{DD^cA^{n-1}}(n+1)\ten \O_{C^*} \\
\overline{DD^{c*}A^{n}}(n)\ten \O_{C^*} \ar[ur]^-{vD^c} & & \overline{D^*D^{c}A^{n}}(n)\ten \O_{C^*} \ar[ul]_-{-uD}\\
& \overline{D^*D^{c*}A^{n+1}}(n-1)\ten \O_{C^*} \ar[ul]^-{uD}\ar[ur]_-{vD^c},
}
\]
which we may regard as a bicomplex. 
By Theorem \ref{Hodgedecomp}, the complex $\tilde{A}^{\bt}(X,\vv)$ decomposes into a direct sum of $\cH^n$'s and total complexes of the bicomplexes above.

Arguing as in Proposition \ref{DeltahalfProp}, we have topological isomorphisms
\begin{align*}
 \Delta^{-1} DD^c\co  \overline{D^*D^{c*}A^{n+1}} &\to    \overline{DD^cA^{n-1}} \\ 
\Delta^{-\half} D\co \overline{D^*D^{c}A^{n}} &\to \overline{DD^cA^{n-1}}\\
\Delta^{-\half}D^c\co \overline{DD^{c*}A^{n}}&\to \overline{DD^cA^{n-1}},
\end{align*}
so the bicomplex above is linearly isomorphic to
\[
 \xymatrix@R=0.5ex{ & \overline{DD^cA^{n-1}}(n+1)\ten \O_{C^*} \\
 \overline{DD^cA^{n-1}}(n)\ten \O_{C^*} \ar[ur]^-{v\Delta^{\half}} & & \overline{DD^cA^{n-1}}(n)\ten \O_{C^*} \ar[ul]_-{-u\Delta^{\half}}\\
& \overline{DD^cA^{n-1}}(n-1)\ten \O_{C^*} \ar[ul]^-{u\Delta^{\half}}\ar[ur]_-{v\Delta^{\half}}.
}       
\]

Since the ideal $(u,v)$ generates $\O_{C^*}$, cohomology of the top level of the associated total complex is just 
\[
  (\overline{DD^cA^{n-1}}/\Delta^{\half}\overline{DD^cA^{n-1}})(n+1)\ten \O_{C^*},      
\]
 while cohomology of the bottom level is $0$. 
Moreover, the map 
\[
 \overline{DD^cA^{n-1}}(n-1) \ten \O_{C^*} \xra{(u,v)}  \overline{DD^cA^{n-1}}(n)^2\ten \O_{C^*}   
\]
is an isomorphism to the kernel of $(v,-u)$, so cohomology of the middle level is isomorphic to 
\[
 (\overline{DD^cA^{n-1}}(n-1)/ \Delta^{\half}\overline{DD^cA^{n-1}}(n-1)) \ten \O_{C^*} \ten \O_{C^*},       
\]
which completes the proof.
\end{proof}

\begin{remark}\label{MTSbad}
 In particular, note that $\tilde{\H}^n(X,\vv)$ is of weights $n, n-1$ in general, unlike $\bar{\tilde{\H}}^n(X,\vv)$ which is pure of weight $n$. In particular, this means that the weight filtration given good truncation cannot define a mixed twistor structure on  $\tilde{\H}^n(X,\vv)$.
\end{remark}

We now have the following generalisation of the principle of two types:
\begin{lemma}\label{p2t2}
As subspaces of  $A^n(X, \vv)\ten O(\SL_2)$,
\begin{align*}
 &\ker \tD \cap \ker \tDc \cap (\overline{ \tD( A^{n-1}(X,\vv)\ten O(\SL_2)) } + \overline{\tDc (A^{n-1}(X,\vv)  \ten O(\SL_2))})\\
&=  \overline{\tD\tDc( A^{n-2}(X,\vv)O(\SL_2))}.       
\end{align*}

\end{lemma}
\begin{proof}
 This follows from Lemma \ref{p2t}, with the same reasoning as \cite[Lemma \ref{mhs2-gl2lemma}]{mhs2}.
\end{proof}

\begin{lemma}[Formality]\label{formallemma2}
The morphisms
\[
 (\bar{\H}_{\tDc}^*(A^*(X,\vv)\ten O(\SL_2)),0) \la (\z_{\tDc}^*(A(X,\vv)\ten O(\SL_2)), \tD) \to (A^*(X,\vv)\ten O(\SL_2), \tD)
\]
induce isomorphisms on reduced cohomology.
\end{lemma}
\begin{proof}
 The proof of \cite[Lemma 2.2]{Simpson} carries over to this generality, using Lemma \ref{p2t2}.
\end{proof}

 Following Corollary \ref{univlocsys} and Remark \ref{DGArk}, the results above can all be regarded as statements about the topological differential graded algebra $\tilde{A}^{\bt}(X, \bE')$.

\subsection{The analytic Hodge filtration on cochains}\label{analyticMHS}

Recall from \S \ref{Evee} that the local system $\bE'$ on $X$ is defined to correspond to the $\pi_1(X,x)$-representation 
 given by left multiplication on $E^J(X,x)'$. 

\begin{proposition}\label{redenrich}
The topological cochain complex $\tilde{A}^{\bt}(X, \bE')$  is equipped with a  continuous circle action,
 satisfying: 
\begin{enumerate}
 \item the $S^1$-action and $\bG_m$-actions  on  $\tilde{A}^{\bt}(X, \bE')$ commute,
\item  the action of $S^1 \subset \Cx^*= S(\R)$ on $C^*$ makes $\tilde{A}^{\bt}(X, \bE')$  into an $S^1$-equivariant sheaf on $C^*$, and 
\item  $-1 \in S^1$ acts as $-1 \in \bG_m$.
\end{enumerate}
\end{proposition}
\begin{proof}
Since $S^1$ acts on $E^J(X,x) $, it acts on $\bE'$, and we denote this action by $v \mapsto t \circledast v$, for $t \in S^1$. 
We may now adapt the proof of \cite[Theorem \ref{mhs2-mhsmal}]{mhs2}, 
defining an $S^1$-action on 
$\sA^*(X, \R)\ten_{\R}\bE' $
by setting $t \boxast (a\ten v) := (t \dmd a) \ten (t^2 \circledast v)$ for $t \in S^1$ and $\dmd$ as in Definition \ref{dmd}. Passing to the completion $\sA^*(X, \bE')$ completes the proof, with continuity following from Proposition \ref{circleX}.
\end{proof}

\begin{remarks} 
If the circle action of Proposition \ref{redenrich} were algebraic, then by
\cite[Lemma \ref{mhs2-tfilenrich}]{mhs2} 
it would correspond to a Hodge filtration on $A^{\bt}(X, \bE')$. Since finite-dimensional circle representations are algebraic, we may regard Proposition \ref{redenrich} as the natural structure of an infinite-dimensional Hodge filtration.

In Proposition \ref{redenrich}, note that we can of course replace $E^J(X,x)$ with any inverse system $B$ of $C^*$-algebra quotients of $E^J(X,x)$ to which the $S^1$-action descends, provided we replace $\bE$ with the local system associated to $B$.

As observed in \S \ref{Evee}, we may  substitute $\vv=\bE$ in Proposition \ref{finecoho} and Lemma \ref{formallemma2}. Note that the resulting isomorphisms are then equivariant with respect to the circle action of Proposition \ref{redenrich}.
\end{remarks}

\subsection{$\SU_2$}\label{SU2sn}

As we saw in Corollary \ref{harmcoho2}, in order to define the  adjoint operator $\tilde{D}^*$ to $\tilde{D}$, it is necessary to pull $\tilde{A}^{\bt}(X,\vv)$ back along the morphism $\row_1 \co \SL_2 \to C^*$. This gives us the complex
\[
 \row_1^*\tilde{A}^{\bt}(X,\vv)= (A^*(X, \vv) \ten O(\SL_2), \tD),
\]
where $\tD=uD+vD^c$, with  adjoint $\tD^*= yD^*-xD^{c*}$. 

This leads us to consider the $*$-structure on $O(\SL_2)$ determined by $u^*=y, v^*=-x$. This implies $x^*=-v, y^*=u$, so
\[
 \begin{pmatrix}u & v \\ x& y \end{pmatrix}^*= \begin{pmatrix}y & -x \\ -v & u \end{pmatrix},
\]
or $A^*= (A^{-1})^t$. 

\begin{lemma}
The real $C^*$-enveloping algebra $C^*(O(\SL_2))$ of the real  $*$-algebra $O(\SL_2)$ is the ring of continuous complex functions $f$ on $\SU_2$ for which
\[
 f(\bar{A}) = \overline{f(A)}.
\]
\end{lemma}
\begin{proof}
A $*$-morphism $O(\SL_2) \to \Cx$ is a matrix $A \in \SL_2(\Cx)$ with $\bar{A}= A^*= (A^{-1})^t$, so $A \in \SU_2$. Thus the Gelfand representation gives $C^*(O(\SL_2))\ten \Cx \cong C(\SU_2, \Cx)$. 

Now, writing $\Gal(\Cx/\R) =\<\tau\>$, and taking  $f \in O(\SL_2)\ten \Cx$ and $A \in \SU_2$, we have $\tau(f)(A) = \overline{f(\bar{A})}$. This formula extends to give a $\Gal(\Cx/\R)$-action on $C(\SU_2, \Cx)$, and 
Lemma \ref{GalCstar} then gives 
\[
 C^*(O(\SL_2))= C(\SU_2, \Cx)^{\tau}.
\]
 \end{proof}

 Note that complex conjugation on $\SU_2$  is equivalent to conjugation by the matrix $\left(\begin{smallmatrix} 0 & 1\\ -1 & 0     \end{smallmatrix}\right)$.

\subsubsection{The Hopf fibration}

The action of $S(\Cx)$ on $\SL_2(\Cx)$ from Definition \ref{rowdef} does not preserve $\SU_2$. However, Lemma \ref{rowlemma} ensures that for the subgroup schemes $\bG_m,S^1 \subset S$, the groups $S^1 = S^1(\R) \subset S^1(\Cx)\cong \Cx^*$ and $S^1\subset \Cx^*\cong \bG_m(\Cx)$ both preserve $\SU_2$.

Thus in the $C^*$ setting, the $S$-action becomes an action of $(S^1\by S^1)/(-1,-1)$ on $\SU_2$, given by 
\[
 (s,t, A)\mapsto  \left(\begin{smallmatrix}  s^{-1} & 0  \\ 0 & s  \end{smallmatrix} \right)A  \left(
\begin{smallmatrix} \Re t & \Im t \\ -\Im t & \Re t \end{smallmatrix} \right)^{-1}.
\]

Moreover, there is a $\Gal(\Cx/\R)$-action on this copy of $S^1 \by S^1$, with the non-trivial automorphism $\tau$ given by $\tau(s,t)= (s^{-1},t)$. The action of $(S^1\by S^1)/(-1,-1)$ is then $\tau$-equivariant. Alternatively, we may characterise our group as $S^1 \by S^1 \subset \Cx^*\by \Cx^* \cong S(\Cx)$, by sending $(s,t)$ to $(st,st^{-1})$. On this group $S^1 \by S^1$, the $\Gal(\Cx/\R)$-action is then given by $\tau(w',w'')= (\overline{w''}, \overline{w'})$.

Now, consider the composition 
\[
 \SL_2 \xra{\row_1} C^* \to [C^*/\bG_m] \cong \bP^1.
\]
On taking Gelfand representations of $C^*$-enveloping algebras, this gives rise to the map
\[
 \SU_2 \to \bP^1(\Cx),
\]
which is just the Hopf fibration $p \co S^3 \to S^2$, corresponding to the quotient by the  action of $S^1 \subset \bG_m(\Cx)$ by diagonal matrices. The action of $\tau$ on $\SU_2$ and on $\bP^1(\Cx)$ is just given by complex conjugation.

\subsubsection{Smooth sections}\label{smoothsnsn}

If we write $\rho_n$ for the weight $n$ action of $\bG_m$ on $\bA^1$, then we may consider the topological vector bundle
\[
  \SU_2\by_{S^1, \rho_n}\Cx      
\]
on $\bP^1(\Cx)$, for the action of $S^1 \subset \bG_m(\Cx)$ above.

\begin{definition}
Define  $\sA^0_{\bP^1}\Cx(n)$ to be the sheaf of smooth sections of    $\SU_2\by_{S^1, \rho_n}\Cx  \to \bP^1(\Cx)$, and write $A^0(\bP^1, \Cx(n)):= \Gamma( \bP^1(\Cx),\sA^0_{\bP^1}\Cx(n))$. Beware that for $n\ne 0$, there is no local system generating $\sA^0_{\bP^1}\Cx(n)$.   
\end{definition}

For $U\subset \bP^1(\Cx)$, observe that $\Gamma( U,\sA^0_{\bP^1}\Cx(n)) $ consists of smooth maps
\[
 f\co p^{-1}(U) \to \Cx       
\]
satisfying $f(\left(\begin{smallmatrix}  s^{-1} & 0  \\ 0 & s  \end{smallmatrix} \right)A)= s^nf(A)$ for all $s \in S^1$.

For the quotient map $q\co C^* \to \bP^1$, 
we may characterise $\Gamma(U, \O_{\bP^1}^{\hol}(n))$ as the space of holomorphic maps
\[
f\co  q^{-1}(U) \to \Cx       
\]
satisfying $f(su,sv)= s^nf(u,v)$ for all $s \in \Cx^*$. The embedding $S^3 \subset C^*(\Cx)$ thus yields
\[
 \O_{\bP^1}^{\hol}(n) \subset \sA^0_{\bP^1}\Cx(n),       
\]
and indeed $\sA^0_{\bP^1}\Cx(n)=  \O_{\bP^1}^{\hol}(n)\ten_{\O_{\bP^1}^{\hol} }\sA^0_{\bP^1}\Cx$.

Now, for the conjugate sheaf $\overline{\O_{\bP^1}^{\hol}(n)}$, note that  $\Gamma(U, \overline{\O_{\bP^1}^{\hol}(n)})$ is the space of anti-holomorphic maps
\[
f\co  q^{-1}(U) \to \Cx       
\]
satisfying $f(su,sv)= \bar{s}^nf(u,v)$ for all $s \in \Cx^*$. Thus we have a canonical embedding 
\[
        \overline{\O_{\bP^1}^{\hol}(n)} \subset\subset \sA^0_{\bP^1}\Cx(-n),
\]
with 
$\sA^0_{\bP^1}\Cx(-n)=  \overline{\O_{\bP^1}^{\hol}(n)}\ten_{\overline{\O_{\bP^1}^{\hol} }}\sA^0_{\bP^1}\Cx$.

Note that the inclusion $O(\SL_2) \subset C(\SU_2, \Cx)$ gives
\[
 u,v \in  A^0(\bP^1, \Cx(1))^{\tau}, \quad   \bar{u},\bar{v} \in  A^0(\bP^1, \Cx(-1))^{\tau}    
\]
and
\[
 O(\SL_2) \subset \bigoplus_{n\in \Z} A^0(\bP^1, \Cx(n))^{\tau}.       
\]

\begin{definition}\label{Ndef}
By \cite[Definition \ref{mhs2-Ndef}]{mhs2}, 
there is a  derivation $N$ on $O(\SL_2)$ given by
 $Nx=u, Ny=v, Nu=Nv=0$, for co-ordinates  $\left(\begin{smallmatrix} u &  v \\ x & y\end{smallmatrix}\right)$ on $\SL_2$. Since this annihilates $u,v$, it is equivalent to  the $O(\SL_2)$-linear map
\[
\Omega(\SL_2/C^*) \to O(\SL_2)        
\]
given by $dx\mapsto u, dy \mapsto v$.
\end{definition}

Note that $N$ has weight $2$, and extends (by completeness) to give $\tau$-equivariant differentials
\[
 N \co  \sA^0_{\bP^1} \Cx(n)\to   \sA^0_{\bP^1} \Cx(n+2).     
\]
Note that $N$ is the composition of the anti-holomorphic differential
\[
 \bar{\pd}_{\bP^1}\co \sA^0_{\bP^1}\Cx(n) \to  \sA^0_{\bP^1}\Cx(n)\ten_{ \overline{\O_{\bP^1}^{\hol} }} \overline{\Omega_{\bP^1}} 
\]
with the canonical isomorphism $\Omega_{\bP^1}\cong \O(-2)$.

\subsubsection{Splittings of the twistor structure}\label{splittingtwistor}

As in \cite[Remark \ref{mhs2-sltrivia}]{mhs2}, 
we can characterise the map $\row_1\co \SL_2 \to C^*$ as the quotient $C^*=[\SL_2/\bG_a]$, where $\bG_a$ acts on $\SL_2$ as left multiplication by $ \left(\begin{smallmatrix} 1 & 0  \\ \bG_a & 1 \end{smallmatrix} \right)$. Here, the $S$-action on $\bG_a$ has $\lambda$ acting as multiplication by $\lambda\bar{\lambda}$.

Therefore the map $q \circ \row_1 \co \SL_2 \to \bP^1$ is given by taking the quotient of $\SL_2$ by the Borel subgroup $B= \bG_m \ltimes \bG_a$, for the action above. The action of $\bG_m$ corresponds to weights, while the action of $\bG_a$ corresponds to the derivation $N$ above, which we regard as the Archimedean monodromy operator as in \cite[\S \ref{mhs2-analogies}]{mhs2}. 

\begin{definition}
 Given a pluriharmonic local system $\vv$, define  $\breve{A}^{\bt}(X, \vv)$ to be the sheaf of $\O_{\bP^1}^{\hol}$-modules associated to the $\bG_m$-equivariant sheaf $\tilde{A}^{\bt}(X, \bE^{\vee}) $ on $C^*$. Explicitly,
\[
 \breve{A}^{\bt}(X, \vv)=  (\bigoplus_n (q_*\tilde{A}^{\bt}(X, \vv))\ten_{\O_{\bP^1}^{\alg}}\O_{\bP^1}^{\hol}(n))^{\bG_m},
\]
so
\[
 \breve{A}^n(X, \vv)= A^n(X,\vv)\ten_{\R} \O_{\bP^1}^{\hol}(n),
\]
with differential $\breve{D}=uD+vD^c$, for $u,v \in \Gamma(\bP^1, \O_{\bP^1}(1))$.
\end{definition}

\begin{definition}
 Write $\mathring{\sA}^{\bt}(X, \vv):= \sA^0_{\bP^1}\ten_{\O_{\bP^1}^{\hol}} \breve{A}^*(X, \vv) $, and observe that this admits an operator
\[
 \breve{D}^c:=-\bar{v}D+uD^c\co \mathring{\sA}^n(X, \vv) \to \mathring{\sA}^{n+1}(X, \vv)(-2).
\]
\end{definition}

Now, applying the map $O(\SL_2) \into   \bigoplus_{n\in \Z} A^0(\bP^1, \Cx(n))^{\tau}$ to 
Lemmas \ref{p2t2},\ref{formallemma2} yields the following:

\begin{lemma}\label{formallemma3}
 As subspaces of   $\mathring{\sA}^n(X, \vv)$,
\[
\ker \breve{D} \cap \ker \breve{D}^c \cap (\overline{ \breve{D} \mathring{\sA}^{n-1}(X, \vv)  } + \overline{\breve{D}^c \mathring{\sA}^{n-1}(X, \vv)(2)}) =  \overline{\breve{D}\breve{D}^c(\mathring{\sA}^{n-2}(X, \vv)(2))}.
\]
Thus the morphisms
\[
 (\bar{\sH}_{\breve{D}^c}^*(\mathring{\sA}^*(X, \vv)),0) \la (\sZ_{\breve{D}^c}^*( \mathring{\sA}^*(X, \vv)), \breve{D}) \to  \mathring{\sA}^n(X, \vv)
\]
induce isomorphisms on reduced cohomology sheaves.
\end{lemma}

Now, $\tilde{A}^{\bt}(X, \bE')$ can be recovered from $\row_1^*\tilde{A}^{\bt}(X, \bE')$ and its nilpotent monodromy operator $N$, and Lemma \ref{formallemma} says  that  $\row_1^*\tilde{A}^{\bt}(X, \bE')$ is equivalent to $\cH^*(X, \bE')\ten O(\SL_2)$ up to reduced quasi-isomorphism.

Under the base change above, we have $N= \bar{\pd}_{\bP^1}$ giving an exact sequence
\[
 0 \to \breve{A}^{\bt}(X, \vv) \to \mathring{\sA}^{\bt}(X, \vv) \xra{N} \mathring{\sA}^{\bt}(X, \vv)(2) \to 0.
\]
In other words, we can recover the topological DGA $\breve{A}^{\bt}(X, \bE')$ from the differential $\bar{\pd}=N$ on the topological DGA $\mathring{\sA}^{\bt}(X, \bE')$, and the latter is just $\bigoplus_n \sA_{\bP^1}( \cH^n(X, \bE')(n))$ up to reduced quasi-isomorphism.

Also note that when we substitute $\vv:=\bE'$ in Lemma \ref{formallemma3}, the morphisms all become equivariant with respect to the circle action of Proposition \ref{redenrich}. This action makes $ \mathring{\sA}^n(X, \bE')$ into an $S^1$-equivariant sheaf over $\bP^1$, where the action on $\bP^1$ is given by $t^2 \in S^1$ sending $(u:v)$ to $t\dmd (u:v) =(au-bv:av+bu)$ for $t=a+ib \in S^1$. 

\section{The twistor family of moduli functors}\label{twistorfamilysn}

\subsection{Pro-Banach algebras on projective space} 

On the complex manifold $\bP^1(\Cx)$, we have a sheaf $\O_{\bP^1}^{\hol}$ of holomorphic functions, which we may regard as a sheaf of pro-Banach algebras. As a topological space, $\bP^1(\Cx)$ is equipped with a $\Gal(\Cx/\R)$-action, the non-trivial element $\tau$ acting on points by complex conjugation. There is also an isomorphism
\[
 \tau^{\sharp}_{\O} \co \tau^{-1} \O_{\bP^1}^{\hol} \to \O_{\bP^1}^{\hol},      
\]
given by 
\[
 \tau^{\sharp}_{\O}(f)(z) = \overline{f(\bar{z})},       
\]
and satisfying
\[
\tau^{\sharp}_{\O}\circ \tau^{-1}(\tau^{\sharp}_{\O}) = \id \co    \O_{\bP^1}^{\hol} \to \O_{\bP^1}^{\hol}.    
\]

\begin{definition}
 Define $\Fr\Alg_{\bP^1,\Cx}$ to be the category of sheaves $\sF$ of unital multiplicatively convex Fr\'echet $\O_{\bP^1}^{\hol}$-algebras (i.e. countable  pro-Banach algebras equipped with a morphism from $\O_{\bP^1}^{\hol}$), quasi-coherent in the sense that the maps
\[
\sF(U)\ten^{\pi}_{\O_{\bP^1}^{\hol}(U)}\O_{\bP^1}^{\hol}(V)  \to \sF(V)    
\]
are isomorphisms for all open subspaces $V \subset U$, where $\ten^{\pi}$ here denotes projective tensor product.
 
Define $\Fr\Alg_{\bP^1,\R}$ to be the category of pairs  $(\sF, \tau^{\sharp}_{\sF})$ for $\sF \in \Fr\Alg_{\bP^1,\Cx}$ and 
\[
\tau^{\sharp}_{\sF} \co \tau^{-1} \sF \to \sF        
\]
an  $(\O_{\bP^1}^{\hol}, \tau^{\sharp}_{\O})$-linear isomorphism satisfying
\[
\tau^{\sharp}_{\sF}\circ \tau^{-1}(\tau^{\sharp}_{\sF}) = \id_{\sF}.        
\]
\end{definition}

Note that for any  m-convex Fr\'echet $k$-algebra $B$, the sheaf $B\ten^{\pi}_k \O_{\bP^1}^{\hol}$ lies in $ \Fr\Alg_{\bP^1,\Cx}$. 
When $k=\R$, the involution $\id_B\ten \tau^{\sharp}_{\O}$ makes $B\ten^{\pi}_k \O_{\bP^1}^{\hol}$ an object of $ \Fr\Alg_{\bP^1,\R}$.
 
The forgetful functor $\Fr\Alg_{\bP^1,\R} \to  \Fr\Alg_{\bP^1,\Cx}$ has a right adjoint, given by $\sF \mapsto \sF \oplus \tau^{-1}\sF $, with the involution $\tau^{\sharp}$ given by swapping summands, and the $\O_{\bP^1}^{\hol}$-structure on $\tau^{-1}\sF $ defined using $\tau_{\O}$.

\subsection{The twistor functors}

\begin{definition}
Define the groupoid-valued functor    $\cR^{\bT,\Cx}_X$ on $\Fr\Alg_{\bP^1,\Cx}$ by letting $\cR^{\bT,\Cx}_X(\sB) $ consist of pairs $(\sT, \tD)$, for $\sA^0_X(\pr_2^{-1}\sB^{\by})$-torsors $\sT$   on $X \by \bP^1(\Cx)$ with $x^*\sT$ trivial as a $\sB^{\by}$-torsor on $\bP^1(\Cx)$, and flat $ud+v\dc$-connections 
\[
 \tD\co \sT \to \sA^1_X\ten_{\sA^0_X}\ad \sT(1).
\]
Here $u,v$ are the basis of $\Gamma(\bP^1_{\R}, \O(1))$ given by the co-ordinates $u,v$ of $C^*$ and the canonical map $C^* \to \bP^1$.

Define the set-valued functor $\oR^{\bT,\Cx}_{X,x}$ on $\Fr\Alg_{\bP^1,\Cx}$ by letting $\oR^{\bT,\Cx}_{X,x}(\sB) $  be the groupoid of triples $(\sT, \tD,f)$, with $(\sT, \tD) \in \cR^{\bT,\Cx}_X(\sB)$ and framing $f \in \Gamma(\bP^1(\Cx), x^*\sT)$.
\end{definition}

\begin{definition}
Define the groupoid-valued functor    $\cR^{\bT}_X$ on $\Fr\Alg_{\bP^1,\R}$  by letting $\cR^{\bT}_X(\sB) $ consist of triples $(\sT, \tD, \tau_{\sT}^{\sharp})$, for $(\sT, \tD) \in \cR^{\bT,\Cx}_X(\sB)$ and isomorphism $\tau^{\sharp}_{\sT} \co (\id_X \by \tau)^{-1}\sT \to \sT$ satisfying the following conditions.  The isomorphism
\[
(\tau^{\sharp}_{\sB}, \tau^{\sharp}_{\sT})\co \sA^1_X(\pr_2^{-1}\sB^{\by}) \by_{\tau^{\sharp}_{\sB}, \sA^1_X(\pr_2^{-1}\tau^{-1}\sB^{\by})}(\id_X \by \tau)^{-1}\sT     \to \sT
\]
must be a morphism of $\sA^1_X(\pr_2^{-1}\sB^{\by})$-torsors, and the diagram
\[
\begin{CD}
 \tau^{-1}\sT @>{\tau^{-1}\tD}>> \sA^1_X\ten_{\sA^0_X}\tau^{-1}\sT(1)\\
@V{\tau^{\sharp}_{\sT}}VV @VV{\tau^{\sharp}_{\sT}}V\\
  \sT @>{\tD}>> \sA^1_X\ten_{\sA^0_X}\sT(1)
 \end{CD}      
\]
must commute.

 Define the  functor $\oR^{\bT}_{X,x}$ on $\Fr\Alg_{\bP^1,\R}$ by letting $\oR^{\bT}_{X,x}(\sB)$ be the groupoid of quadruples $(\sT, \tD, \tau_{\sT}^{\sharp},f)$, for $(\sT, \tD, \tau_{\sT}^{\sharp})$ in $\cR^{\bT}_X(\sB) $ and 
\[
 f \in  \Gamma(\bP^1(\Cx), x^*\sT)^{ \tau_{\sT}^{\sharp}}     
\]
a framing. 
\end{definition}

\begin{remark}
 Observe  that the groupoids $ \oR^{\bT,\Cx}_{X,x}(\sB),\oR^{\bT}_{X,x}(\sB)$ are equivalent to  discrete groupoids, so we will regard them as set-valued functors (given by isomorphism classes of objects).  Also note that $\cR^{\bT,\Cx}_X(\sB)$ and $\cR^{\bT}_X(\sB)$ are equivalent to the groupoid quotients $[\oR^{\bT,\Cx}_{X,x}(\sB)/\Gamma(\bP^1(\Cx), \sB^{\by})]$ and $[\oR^{\bT}_{X,x}(\sB)/\Gamma(\bP^1(\Cx), \sB^{\by})^{ \tau_{\sT}^{\sharp}}]$ respectively, with the action given by changing the framing.
\end{remark}

The following is straightforward:
\begin{lemma}
The functors $\oR^{\bT,\Cx}_{X,x}$ and $\cR^{\bT,\Cx}_{X}$ can be recovered from  $\oR^{\bT}_{X,x}$ and $\cR^{\bT}_{X}$, respectively, via isomorphisms
\[
 \oR^{\bT,\Cx}_{X,x}(\sB)\cong  \oR^{\bT}_{X,x}(\sB \oplus \tau^{-1}\sB)  \quad   \cR^{\bT,\Cx}_{X}(\sB)\cong  \cR^{\bT}_{X}(\sB \oplus \tau^{-1}\sB).  
\]
\end{lemma}

Note that the proof of Lemma \ref{tensorstr0} carries over to give canonical comultiplication
\[
 \oR^{\bT}_{X,x}(\sB_1) \by \oR^{\bT}_{X,x}(\sB_2) \to \oR^{\bT}_{X,x}(\sB_1\ten^{\pi} \sB_2).
\]

\subsection{Universality and $\sigma$-invariant sections}

\begin{proposition}\label{twistoruniversal}
The functors $\cR^{\dR}_X$ and $\cR^{\Dol}_X$ can be recovered from $\cR^{\bT}_{X}$ and  $\cR^{\bT,\Cx}_{X}$, respectively. Likewise, $\oR^{\dR}_{X,x}$ and $\oR^{\Dol}_{X,x}$ can be recovered from $\oR^{\bT}_{X,x}$ and  $\oR^{\bT,\Cx}_{X,x}$. 
\end{proposition}
\begin{proof}
Given a point $p \in \bP^1(\Cx)$ and a complex m-convex Fr\'echet algebra $B$, we may regard the skyscraper sheaf $p_*B$ as an object of $ \Fr\Alg_{\bP^1,\Cx}$. For a real m-convex Fr\'echet algebra $B$ and $p \in \bP^1(\R)$, we may regard $p_*B\ten \Cx$ as an object of  $\Fr\Alg_{\bP^1,\R}$, with $\tau^{\sharp}$ given by complex conjugation.

Now, just observe that on pulling back to $(1:0) \in \bP^1(\R)$, the differential $ud+v\dc$ is just $d$. At $(1:-i) \in \bP^1(\R)$, we have $ud+v\dc= d-i\dc= 2\pd$. Uncoiling the definitions, this gives
\[
 \oR^{\dR}_{X,x}(B) =  \oR^{\bT}_{X,x}((1:0)_*B),\quad  \oR^{\Dol}_{X,x}(B) =  \oR^{\bT,\Cx}_{X,x}((1:-i)_*B),    
\]
and similarly for $\cR$.
\end{proof}

\begin{lemma}\label{torsorsections}
For a real m-convex  Fr\'echet algebra $B$, the groupoid  $\cR^{\bT}_{X}(B\ten_{\R}^{\pi} \O_{\bP^1}^{\hol})$ is equivalent to the groupoid of triples $(\sP, D,E)$, where $\sP$ is an $\sA^0_X(B^{\by})$-torsor on $X$, $D,E\co \sP \to \sA^1_X\ten_{\sA^0_X}\ad \sP$ are flat $d$- and $\dc$-connections, respectively, and $DE+ED=0$.  
\end{lemma}
\begin{proof}
Take an object $(\sT, \tD) \in  \cR^{\bT}_{X}(B\ten_{\R} \O_{\bP^1}^{\hol})$. Triviality of the $(B\ten_{\R}^{\pi} \O_{\bP^1}^{\hol})^{\by}$-torsors $x^*\sT$ for $x \in X$ ensures that $\sT$ must be of the form $\sA^0_X(B\ten_{\R}^{\pi}B)^{\by}\by_{\sA^0_X(B^{\by})}\sP$ for some $\sP$ as above.

Now, since $\Gamma(\bP^1,\O_{\bP^1}^{\hol}(1))^{\tau^{\sharp}}= \R u \oplus \R v$,  the connection $\tD$ can be regarded as a  $ud+v\dc$ connection
\[
 \tD\co \sP \to   (\sA^1_X\ten_{\sA^0_X}\ad \sP)\ten_{\R}(\R u \oplus \R v),    
\]
which we write as $uD+vE$.

Flatness of $\tD$ is now a statement about $\Gamma(\bP^1,\O_{\bP^1}^{\hol}(2))^{\tau^{\sharp}}= \R u^2 \oplus \R uv \oplus \R v^2 $, with the $u^2$ and $v^2$ terms giving flatness of $D$ and $E$, and the $uv$ term giving the condition $[D,E]=0$.
\end{proof}

\begin{definition}
Define the involution $\tau\sigma_{\bP}$ of the polarised scheme $(\bP^1, \O_{\bP^1}(1))$ to be the map induced by the action of $i \in S(\R)$ on $C^*$ from Definition \ref{cdef}. In particular, $\tau\sigma_{\bP}(u:v)= (v:-u)$.  
\end{definition}

\begin{remark} 
On \cite[p12]{MTS}, the co-ordinate system $(u+iv:u-iv)$ on $\bP^1(\Cx)$ is used, and antiholomorphic involutions $\sigma, \tau$ are defined. In our co-ordinates, these become $\sigma(u:v)= (\bar{v}, -\bar{u})$ and $\tau(u:v)= (\bar{u}, \bar{v})$. This justifies the notation $\tau\sigma$ used above. Also note that the $\bG_{m,\Cx}$-action on $\bP^1_{\Cx}$ given in \cite[p4]{MTS} is just the complex form of our circle action on $\bP^1_{\R}$ from \S\S \ref{analyticMHS} and \ref{splittingtwistor}.
\end{remark}

\begin{definition}
 Given a real Banach algebra $B$, define the involution $\tau\sigma'$ of $\cR^{\bT}_{X}(B\ten_{\R}^{\pi} \O_{\bP^1}^{\hol})$ by sending the pair $(\sT, \tD)$ to $(\tau\sigma_{\bP}^{-1}\sT, J\tau\sigma_{\bP}^{-1}\tD)$. Note that  this is well-defined because $\tau\sigma_{\bP}^{-1}\tD $ is a $\tau\sigma_{\bP}(ud+v\dc)= (vd-u\dc)$-connection, and $Jd=\dc, J\dc=-d$,  so $J\tau\sigma_{\bP}^{-1}\tD $ is a $(ud+v\dc)$-connection.
\end{definition}

\begin{definition}
Given a   $C^*$-algebra $B$, define the Cartan involution $C$ of $B^{\by}$ to be given by $C(g)= (g^{-1})^*$. 
 Note that this induces a Lie algebra involution  $\ad C\co b \mapsto -b^*$ on the tangent space $B$ of $B^{\by}$. 

If $B=A\ten \Cx$ for  a real $C^*$-algebra $A$, we write $\tau$ for complex conjugation, so $C\tau$ is the involution $C\tau(g)= (\bar{g}^{-1})^*$.  Note that $\ad C\tau$ is  the $\Cx$-linear extension of $\ad C$ on $A$.
\end{definition}

Since $\cR^{\bT}_{X}(\sB) $ only depends on the group of units $\sB^{\by}$ of $\sB$, and its tangent Lie algebra $\sB$, the Cartan involution  induces an involution $C\tau$ of $\cR^{\bT}_{X}(B\ten_{\R}^{\pi} \O_{\bP^1}^{\hol}) $.

\begin{definition}
 Given a real Banach algebra $B$, define the involution $\sigma$ of $\cR^{\bT}_{X}(B\ten_{\R}^{\pi} \O_{\bP^1}^{\hol})$ by $\sigma:= (C\tau)(\tau\sigma')$.
\end{definition}

\begin{proposition}\label{sigmasections}
For a real $C^*$-algebra $B$, there is a canonical isomorphism
\[
 \oR^{\bT}_{X,x}(\sO_{\bP^1}^{\hol}(B))^{\sigma} \cong \oR^J_{X,x}(B).
\]
\end{proposition}
\begin{proof}
Since $\oR^{\bT}_{X,x} $ is the groupoid fibre of $\cR^{\bT}_X\to \cR^{\bT}_{\{x\}}$ over the trivial torsor, Lemma \ref{torsorsections} shows that an object $(\sT, \tD,f)$ of $\oR^{\bT}_{X,x}(\sO_{\bP^1}^{\hol}(B))$ is a quadruple $(\sP, D,E,f)$ for $(\sP,D,E)$ as in that lemma, and $f$ our framing. We therefore begin by describing the $\sigma$-action on such data. For $(\sP, D,E,f)$ to be $\sigma$-invariant, we must have an isomorphism $\alpha\co (\sP, D,E,f) \to \sigma(\sP, D,E,f) $.

The torsor $\sP$ maps under $\sigma$ to $C(\sP)$, with $f\in x^*\sP$ mapping to $C(f)$ (since $\sigma'$ and $\tau$ affect neither). The isomorphism $\alpha$ then gives $\alpha\co \sP \to C(\sP)$ such that $\alpha(f) = C(f) \in x^*C(\sP)$. Let $U(\sP) \subset \sP$ consist of sections $q$ with $\alpha(q)= C(q)$. This is non-empty (since its fibre at $x$ contains $f$), so it must be an $\sA^0_X(U(B))$-torsor, noting that $U(B)$ is the group of $C$-invariants in $B^{\by}$. Moreover $\sP=  \sA^0_X(B^{\by})\by_{\sA^0_X(U(B))}U(\sP)$.

Meanwhile, $\tau\sigma'(uD+vE) = vJD -uJE$, so $\tau\sigma'(D,E)= (-JE, JD)$. Thus the isomorphism $\alpha$ gives
\[
 D|_{\sQ}= -JCE, \quad E|_{\sQ}= JCD.       
\]
In other words, $E=D^c$ and $E^c=-D$ (which are equivalent conditions). 

Thus $\oR^{\bT}_{X,x}(\sO_{\bP^1}^{\hol}(B))^{\sigma}$ is equivalent to the groupoid  of triples $(U(\sP), D,f)$, with $D$ flat and $[D,D^c]=0$.
\end{proof}

\begin{remark}\label{sigmark}
When $B= \Mat_n(\Cx)$, this shows that framed pluriharmonic local systems correspond to framed $\sigma$-invariant sections of the twistor functor. Without the framings, this will not be true in general, since a $\sigma$-invariant section of the coarse moduli  space will give a non-degenerate bilinear form which need not be positive definite.  Note that for $U \subset \bP^1$, the set of isomorphism classes of $\cR^{\bT, \Cx}_{X}(\Mat_n(\sO_U^{\hol}))$ is the set of sections over $U$ of the twistor space $TW \to \bP^1$ of  \cite[\S 3]{simfil}.  
\end{remark}

\begin{remark}\label{twistordgrk}
Although we have seen that  $\cR^{\bT}_X$ together with its comultiplication encodes all the available information about twistor structures on moduli spaces of local systems, it does not carry information about higher homotopy and cohomology groups. There is, however, a natural extension of $\cR^{\bT}_X$ to differential graded m-convex Fr\'echet algebras, 
by analogy with \cite{htpy,mhs2}. This would involve taking $\tD$ to be a hyperconnection $\tD\co \cT_0 \to \prod_n \sA^{n+1}_X\ten_{\sA^0_X}\ad \sT_n(n+1)$. The structures of \S \ref{cohosn} can all be recovered from this functor.
\end{remark}

\subsection{Topological twistor representation spaces}

In this section, we will show that by considering continuous homomorphisms rather than $*$-homomorphisms, we can describe the entire semisimple locus  of the twistor family  from  $(E^J_{X,x})_{\PN}$, rather than just the $\sigma$-equivariant sections.

Given a point $(a: b) \in \bP^1(\Cx)$   and a complex Banach algebra $B$, we can generalise the construction of Proposition \ref{twistoruniversal} and consider the set  $\oR^{\bT,\Cx}_{X,x}((a:b)_*B)$. This consists of torsors with flat $(ad+b\dc)$-connections.

\begin{definition}
Define   $\oT_{X,x,n}:= \coprod_{(a:b) \in \bP^1(\Cx)} \oR^{\bT,\Cx}_{X,x}((a:b)_*\Mat_n(\Cx))$. 
\end{definition}

Note that $\oT_{X,x,n}$ inherits a $\Gal(\Cx/\R)$-action from $\oR^{\bT,\Cx}_{X,x} $. We can also  give $\oT_{X,x,n}$ a complex analytic structure, by saying that a map $f\co U \to \oT_{X,x}(B)$ from an analytic space $U$ consists of an analytic map $f\co U \to \bP^1(\Cx)$ together with an element of $\oR^{\bT,\Cx}_{X,x}(f_*(\Mat_n\O_U))$. We will now investigate the underlying topological structure.

\begin{remark}\label{cfdh}
The adjoint action of $\GL_n(\Cx)$ on  $\oT_{X,x,n}$ is continuous, and indeed compatible with the complex analytic structure. This allows us to consider the coarse quotient $\oT_{X,x,n}//\GL_n(\Cx) $, which is the Hausdorff completion of the topological quotient, equipped with a natural complex analytic structure over $\bP^1(\Cx)$. A straightforward calculation shows that 
this coarse moduli space   is precisely the Deligne-Hitchin twistor space, as constructed in \cite{hitchin} and described in \cite{simpsonwgt2} \S 3. 
\end{remark}

Now,  Proposition \ref{sigmasections} induces a map $\pi_{\bT}\co \oR^J_{X,x}(\Mat_n(\Cx)) \by \bP^1(\Cx) \to \oT_{X,x,n}$. For an explicit characterisation, note that an $(ad+b\dc)$-connection $\tD$ lies in $ \oR^J_{X,x}(\Mat_n(\Cx))$ if and only if 
\[
 [\tD, JC\tD]=0,        
\]
and that $JC\tD$ is a $(-\bar{b}d+\bar{a}\dc) = \sigma_{\bP}^{-1}(ad+b\dc)$-connection.

\begin{definition}
 Given a flat $(ad+b\dc)$-connection $\tD$ on a finite-dimensional $\C^{\infty}$ vector bundle $\sV$ on $X$ for $(a:b) \in \bP^1(\Cx)- \{\pm i\}$, we say that $(\sV,\tD)$ is \emph{semisimple} if the local system $\ker \tD$ is so. 
\end{definition}

\begin{definition}
Define  $\oT_{X,x,n}^{\st}\subset \oT_{X,x,n}$ by requiring that the fibre over any point of $\bP^1(\Cx)- \{\pm i\}$ consist of the semisimple objects, that the fibre over $i$ be $\oR^{\Dol,\st}_{X,x}(\Mat_n(\Cx))$ (\S \ref{Dolsn}), and that the fibre over $-i$ be its conjugate.

We give $\oT_{X,x,n}^{\st}$ the subspace topology, so a map $K \to \oT_{X,x,n}^{\st}$ is continuous if the projection $f \co K \to \bP^1(\Cx)$ is so, and the map lifts to an element of $ \oR^{\bT,\Cx}_{X,x}(f_*C(K, \Mat_n(\Cx)))$.
\end{definition} 

\begin{theorem}\label{HTtopthm}
For any positive integer $n$, there is a natural homeomorphism $\pi_{\bT,\st}$ over $\bP^1(\Cx)$ between the space 
$\Hom_{\pro(\Ban\Alg)}(E^J_{X,x},\Mat_n(\Cx)) \by \bP^1(\Cx)$ with the topology of pointwise convergence, and  the space  $\oT_{X,x,n}^{\st}$.   
\end{theorem}
\begin{proof}
The homeomorphism is given on the fibre over $(a:b) \in \bP^1(\Cx)$ by $\pi_{\bT,\st}(U(\sP), D,f) = (\sP, aD+bD^c, f)$. The proofs of Theorems \ref{Hsstopthm} and \ref{Hsttopthm} adapt to show that $\pi_{\bT,\st}$ induces  homeomorphisms on fibres over $\bP^1(\Cx)$, and in particular is an isomorphism on points. The same arguments also show that $\pi_{\bT,\st} $ and $\pi_{\dR} \circ  \pi_{\bT,\st}^{-1}\co \pi_{\bT,\st}\to \oR^{\dR,\ss}_{X,x}(\Mat_n(\Cx))$ are continuous, so the result follows from  Theorem \ref{Hsstopthm}.    
\end{proof}

\begin{definition}\label{FDTcat}
 Let $\FD\oT_{X,x}^{\st}$ be the category of pairs $(V,p, (a:b))$ for $V\in \FD\Vect$ and  $(p,(a:b))  \in \oT_{X,x,n}^{\st}$ where $n=\dim V$. Morphisms are defined by adapting the formulae of Definition \ref{FDDRcat}.
Write  $\eta_x^{\bT,\st}\co \FD\oT^{\st}_{X,x} \to \FD\Vect\by \bP^1(\Cx)$ for the fibre functors $(V,p, (a:b)) \mapsto (V, (a:b))$.
\end{definition}

\begin{proposition}\label{PNTprop}
The $C^*$-algebra $ (E^J_{X,x})_{\PN}\hten_{\R} C(\bP^1(\Cx), \Cx)$ is isomorphic to the ring of continuous   endomorphisms of $\eta_x^{\bT, \st}$, and this isomorphism is $\Gal(\Cx/\R)$-equivariant.
\end{proposition}
\begin{proof}
The proof of Proposition \ref{PNssprop} carries over, replacing Theorem \ref{Hsstopthm}  with Theorem \ref{HTtopthm}, and Lemma \ref{PNlemma2} with Lemma \ref{PNlemma2base}.
\end{proof}

\begin{definition}\label{stTdef}
 Given a    $k$-normal real  $C^*$-algebra $B$ over $C(\bP^1(\Cx), \Cx)^{\Gal(\Cx/\R)}$, we may regard $B$ as an $\O_{\bP^1}^{\hol}$-algebra via the inclusion of holomorphic functions in continuous functions. 
Then define $ \oR^{\bT,\st}_{X,x}(B)\subset \oR^{\bT}_{X,x}(B)$  to  be the subspace consisting  of those $p$ for which
\[
 (\psi(p),(a:b)) \in \oT_{X,x,k}^{\st}
\]
for all $(a:b) \in \bP^1(\Cx)$ and  $\psi \co B \to \Mat_k(\Cx) $ with $\psi|_{C(\bP^1(\Cx), \Cx)^{\Gal(\Cx/\R)}}= \ev_{(a:b)}\id$.  
\end{definition}

\begin{corollary}\label{PNrepT}
For any $k$-normal real $C^*$-algebra $B$ over $C(\bP^1(\Cx), \Cx)^{\Gal(\Cx/\R)}$, there is a natural isomorphism between $ \oR^{\bT,\st}_{X,x}(B)^{\Gal(\Cx/\R)}$ and the set of continuous algebra homomorphisms $E^J_{X,x} \to B$.
\end{corollary}
\begin{proof}
The proof of Corollary \ref{PNrepss} carries over, replacing Proposition \ref{PNssprop} with  Proposition \ref{PNTprop}.
 \end{proof}

\bibliographystyle{alphanum}
\addcontentsline{toc}{section}{Bibliography}
\bibliography{references.bib}
\end{document}